\documentclass[11pt,a4paper]{article}
\usepackage[margin=1in]{geometry}
\usepackage{xspace}
\usepackage{amsmath,amssymb,amsthm}
\usepackage{graphics}
\usepackage{graphicx}
\usepackage{hyperref}
\usepackage{tikz}
\usepackage{paralist}
\usepackage{wrapfig}
\usepackage{bibunits}

\newtheorem{theorem}{Theorem}
\newtheorem{lemma}{Lemma}
\newtheorem{corollary}{Corollary}

\defaultbibliography{abbrv}
\defaultbibliography{regular_augmentation_shortAndProofs}

\def\comment#1{}

\def\withcomments{
  \newcounter{mycommentcounter}
  \def\comment##1{\refstepcounter{mycommentcounter}%
   \ifhmode%
    \unskip%
    {\dimen1=\baselineskip \divide\dimen1 by 2 %
      \raise\dimen1\llap{\tiny -\themycommentcounter-}}\fi%
    \marginpar{\renewcommand{\baselinestretch}{0.8}%
      \footnotesize [\themycommentcounter]: \raggedright ##1}}
  }
\withcomments

\newcommand{\rephrase}[3]{\noindent\textbf{#1~#2.~}\emph{#3}}

\newcommand{\myproblem}[3]{\medskip\noindent\textbf{Problem:}
  #1\\\noindent Instance: #2\\\noindent Task: #3\smallskip}

\newcommand*\circled[1]{ % 
  \protect\tikz[baseline=(char.base)]{ % 
    \protect\node[shape=circle,draw,inner sep=0.2pt] (char) {#1};}} %
    
\newcommand{\0}[1]{#1^{\tiny \ensuremath{\circled{0}}}}
\newcommand{\1}[1]{#1^{\tiny \ensuremath{\circled{1}}}}

\newcommand{\m}[1]{#1^{\footnotesize \ensuremath{\circled{-}}}}
\newcommand{\p}[1]{\ensuremath{#1^\oplus}}
\newcommand{\s}[1]{#1^{\tiny \ensuremath{\circled{$\star$}}}}
\newcommand{\ii}[1]{#1^{\tiny \ensuremath{\circled{i}}}}

\newcommand{\sV}{\s{V}}
\newcommand{\sX}{\s{X}}
\newcommand{\mW}{\m{W}}
\newcommand{\pW}{\p{W}}

\newtheorem{observation}{Observation}
\newtheorem{selectionRule}{Rule}
\newtheorem{condition}{Condition}

\newcommand{\planaug}{\textsc{P}\aug}
\newcommand{\aug}{\textsc{RA}\xspace}
\newcommand{\faug}{\textsc{FE}\aug}

\newcommand{\sat}{\textsc{3Sat}\xspace}
\newcommand{\psat}{\textsc{Planar}\sat}
\newcommand{\mpsat}{\textsc{Monotone}\psat}

\newcommand{\F}{\ensuremath{\mathcal{F}}}
\newcommand{\kmax}{\ensuremath{k_{\max}}}

\newcommand{\Vin}{V_{\mathrm{in}}} 
\newcommand{\Vbound}{V_b}
\newcommand{\preA}{\widetilde{A}} 
\newcommand{\prea}{\widetilde{a}_f}
\newcommand{\predem}{\widetilde{d}_f} 

\title{Cubic Augmentation of Planar Graphs}%
\author{Tanja Hartmann\thanks{Partially supported by the DFG under grant WA 654/15 within the Priority Programme "Algorithm Engineering".}, Jonathan Rollin, Ignaz Rutter}%
\date{Karlsruhe Institute of
  Technology (KIT)\\\texttt{firstname.lastname@kit.edu}}%

\begin{document}

\maketitle

%\vspace{-2ex}
\begin{abstract}
  In this paper we study the problem of augmenting a planar graph such
  that it becomes 3-regular and remains planar.  We show that it is
  NP-hard to decide whether such an augmentation exists.  On the other
  hand, we give an efficient algorithm for the variant of the problem
  where the input graph has a fixed planar (topological) embedding
  that has to be preserved by the augmentation.  We further generalize
  this algorithm to test efficiently whether a 3-regular planar
  augmentation exists that additionally makes the input graph
  connected or biconnected. If the input graph should become even 
  triconnected, we show that the existence of a 3-regular planar augmentation is again NP-hard to decide.
\end{abstract}

\section{Introduction}
\label{sec:introduction}

An \emph{augmentation} of a graph~$G=(V,E)$ is a set~$W \subseteq E^c$
of edges of the complement graph.  The \emph{augmented graph} $G'=(V,E
\cup W)$ is denoted by~$G+W$.  We study several problems where the
task is to augment a given planar graph to be 3-regular while
preserving planarity.  The problem of augmenting a graph with the goal
that the resulting graph has some additional properties is a
well-studied problem and has applications in network
planning~\cite{et-ap-76}.  Often the goal is to increase the
connectivity of the graph while adding few edges.  Nagamochi and
Ibaraki~\cite{ni-gcaam-02} study the problem making a graph
biconnected by adding few edges.  Watanabe and
Nakamura~\cite{wn-ecap-87} give an~$O(c \min\{c,n\}n^4(cn+m))$
algorithm for minimizing the number of edges to make a graph
$c$-edge-connected.  The problem of biconnecting a graph at minimum
cost is NP-hard, even if all weights are
in~$\{1,2\}$~\cite{ni-gcaam-02}.  Motivated by graph drawing
algorithms that require biconnected input graphs, Kant and
Bodlaender~\cite{kb-pgap-91} initiated the study of augmenting the
connectivity of planar graphs, while preserving planarity.  They show
that minimizing the number of edges for the biconnected case is
NP-hard and give efficient 2-approximation algorithms for both
variants.  Rutter and Wolff~\cite{rw-acpgg-12} give a corresponding
NP-hardness result for planar 2-edge connectivity and study the
complexity of geometric augmentation problems, where the input graph
is a plane geometric graph and additional edges have to be drawn as
straight-line segments.  Abellanas et al.~\cite{aghtu-acgg-08},
T{\'o}th~\cite{t-capsl-08} and Al-Jubeh et al.~\cite{airst-ecap-09}
give upper bounds on the number of edges required to make a plane
straight-line graph $c$-connected for~$c=2,3$.
For a survey on plane geometric graph augmentation see~\cite{ht-pga-12}.

We study the problem of augmenting a graph to be 3-regular while
preserving planarity.  In doing so, we additionally seek to raise the
connectivity as much as possible.  Specifically, we study the
following problems.

\myproblem{{\sc Planar 3-Regular Augmentation} (\planaug)}{Planar graph~$G=(V,E)$}{Find an augmentation~$W$
  such that~$G+W$ is 3-regular and planar.}

\myproblem{{\sc Fixed-Embedding Planar 3-Regular Augmentation}
  (\faug)}{Planar graph~$G=(V,E)$ with a fixed planar (topological)
  embedding}{Find an augmentation~$W$ such that~$G+W$ is 3-regular,
  planar, and $W$ can be added in a planar way to the fixed embedding
  of~$G$.}

Moreover, we study \emph{$c$-connected} \faug, for $c=1,2,3$, where the goal is
to find a solution to \faug, such that the resulting graph
additionally is $c$-connected. 

\paragraph{Contribution and Outline.}

Using a modified version of an NP-hardness reduction by Rutter and
Wolff~\cite{rw-acpgg-12}, we show that \planaug is NP-hard; the proof
is postponed to Section~\ref{sec:hardness}.

%\vspace{-1ex}

\newcommand{\theoremNpc}{\planaug is NP-complete, even if the input
  graph is biconnected.}

\begin{theorem}
  \label{thm:npc}
  \theoremNpc
\end{theorem}
%
%\vspace{-1ex}
%
Our main result is an efficient algorithm for \faug and $c$-connected
\faug for~$c=1,2$.  We note that Pilz~\cite{p-acg-12} has
simultaneously and independently studied the planar 3-regular
augmentation problem. He showed that it is NP-hard and posed the question on 
the complexity if the embedding is fixed.  Our hardness proof strengthens his result
(to biconnected input graphs) and our algorithmic results answer his
open question.
We further prove that for $c=3$ $c$-connected $\faug$ is again NP-hard.

We introduce basic notions used throughout the paper in
Section~\ref{sec:preliminaries}.  We present our results on \faug
in Section~\ref{sec:planar-fixed}.  The problem is
equivalent to finding a \emph{node assignment} that assigns the
vertices with degree less than~3 to the faces of the graph, such that
for each face~$f$ an augmentation exists that can be embedded in~$f$
in a planar way and raises the degrees of all its assigned vertices
to~$3$.  We completely characterize these assignments and show that
their existence can be tested efficiently.  We strengthen our
characterizations to the case where the graph should become
$c$-connected for~$c=1,2$ in Section~\ref{sec:planar-fixed-conn} and
show that our algorithm can be extended to incorporate these
constraints.  
In Section~\ref{sec:tricon} and Section~\ref{sec:hardness} we provide the hardness proofs
for $3$-connected $\faug$ and $\planaug$.

\section{Preliminaries}
\label{sec:preliminaries}

A graph~$G=(V,E)$ is \emph{3-regular} if all vertices have degree~3.
It is a \emph{maxdeg-3 graph} if all vertices have at most degree~3.
For a vertex set~$V$, we denote by~$\0V, \1V$ and $\2V$ the set of
vertices with degree~0,1 and 2, respectively.  For convenience, we
use~$\sV = \0V \cup \1V \cup \2V$ to denote the set of vertices with
degree less than~3.  Clearly, an augmentation~$W$ such that~$G+W$ is
3-regular must contain~$3-i$ edges incident to a vertex in $\ii{V}$.  We
say that a vertex~$v \in \ii V$ has~$3-i$ \emph{(free) valencies} and
that an edge of an augmentation incident to~$v$ \emph{satisfies} a
valency of~$v$.  Two valencies are adjacent if their vertices are
adjacent.

Recall that a graph~$G$ is \emph{connected} if it contains a
path between any pair of vertices, and it
is~\emph{$c$-(edge)-connected} if it is connected and removing any
set of at most~$c-1$ vertices (edges) leaves $G$ connected.  A
2-connected graph is also called \emph{biconnected}.  We note that the
notions of $c$-connectivity and $c$-edge-connectivity coincide on
maxdeg-3 graphs.  Hence a maxdeg-3 graph is biconnected if and only if
it is connected and does not contain a \emph{bridge}, i.e., an edge
whose removal disconnects the graph.

A graph is \emph{planar} if it admits a \emph{planar embedding} into
the Euclidean plane, where each vertex (edge) is mapped to a distinct
point (Jordan curve between its endpoints) such that curves
representing distinct edges do not cross.  A planar embedding of a
graph subdivides the Euclidean plane into \emph{faces}.  When we seek
a planar augmentation preserving a fixed embedding, we require that
the additional edges can be embedded into these faces in a planar way.

\section{Planar 3-Regular Augmentation with Fixed Embedding}
\label{sec:planar-fixed}

In this section we study the problem \faug of deciding for a
graph~$G=(V,E)$ with fixed planar embedding, whether there
exists an augmentation~$W$ such that~$G+W$ is 3-regular and the
edges in~$W$ can be embedded into the faces of $G$ in a planar way.

An augmentation~$W$ is \emph{valid} only if the endpoints
of each edge in~$W$ share a common face in~$G$.  We assume that a
valid augmentation is associated with a (not necessarily planar)
embedding of its edges into the faces of~$G$ such that each edge is
embedded into a face shared by its endpoint.  A valid augmentation is
\emph{planar} if the edges can be further embedded in a planar way
into the faces of~$G$.  

Let $F$ denote the set of faces of $G$ and recall that~$\sV$ is the
set of vertices with free valencies.  A \emph{node assignment} is a
mapping $A \colon \sV \to F$ such that each $v\in \sV$ is incident to
$A(v)$.  Each valid 3-regular augmentation $W$ induces a node
assignment by assigning each vertex $v$ to the face where 
its incident edges in $W$ are embedded: this is well-defined  
since vertices in $\0V\cup \1V$ are incident to a single face.
A node assignment is
\emph{realizable} if there exists a valid augmentation that induces
it.  It is \emph{realizable in a planar way} if it is induced by some
planar augmentation.  We also call the corresponding augmentation a
\emph{realization}.
%--------------- neu -------------
A realizable node assignment can be found 
efficiently by computing a matching in the subgraph of~$G^c$ that
contains edges only between vertices that share a common face.  The
existence of such a matching is a necessary condition for the
existence of a planar realization. The main result of this section
is that this condition is also sufficient.

Both valid augmentations and node assignments are local by nature, and
can be considered independently for distinct faces.  Let~$A$ be a node
assignment and let~$f$ be a face.  We denote by~$V_f$ the vertices
that are assigned to~$f$.  We say that $A$ is \emph{realizable
  for $f$} if there exists an augmentation~$W_f \subseteq \binom{V_f}{2}$ such that in~$G+W_f$ all vertices of~$V_f$ have degree~3.  It is \emph{realizable for~$f$ in a planar way} if additionally $W_f$ can be embedded
in~$f$ without crossings.  We call the corresponding augmentations
\emph{(planar) realizations for~$f$}.  
The following lemma is obtained by glueing (planar) realizations for all faces.

\newcommand{\lemmaGlobalLocal}{A node assignment is realizable (in a
  planar way) for a graph $G$ if and only if it is realizable (in a
  planar way) for each face $f$ of $G$.}

\begin{lemma}
  \label{lem:global-local}
  \lemmaGlobalLocal
\end{lemma}
% -------------- Proof eingefuegt -----------------
\begin{proof}
  Consider a node assignment $A$.  If $A$ is realizable (in a planar way), there exists a corresponding valid (planar) augmentation~$W$.  Then for each face~$f$ the set~$W_f \subseteq W$ of edges embedded in~$f$ forms a (planar) realization for~$f$.  Conversely, assume that~$A$ is realizable (in a planar way) for each face~$f$.  Then for each face~$f$ there is a corresponding (planar) realization~$W_f$ of~$A$ for~$f$.  Hence $W :=\bigcup_{f\in F} W_f $ is a valid (planar) augmentation that realizes $A$.
\end{proof}
% ---------- ende Proof -------------------

Note that a node assignment induces a unique corresponding assignment
of free valencies, and we also refer to the node assignment
as assigning free valencies to faces.
In the spirit of the notation $G + W$ we use $f + W_f$ to denote the
graph $G+W_f$, where the edges in~$W_f$ are embedded into the
face~$f$.  If $W_f$ consists of a single edge $e$, we write $f+e$.
For a fixed node assignment $A$ we sometimes consider an augmentation
$W_f$ that realizes $A$ for~$f$ only in parts by allowing that some
vertices assigned to~$f$ have still a degree less than~3 in $f + W_f$.
We then seek an augmentation~$W_f'$ such that~$W_f \cup W_f'$ forms a
realization of~$A$ for~$f$.  We interpret~$A$ as a node assignment
for~$f+W_f$ that assigns to~$f$ all vertices that were originally
assigned to $f$ by $A$ and do not yet have degree~3 in~$f+W_f$.
Observe that in doing so, we still assign to the faces of~$G$ but when
considering free valencies and adjacencies, we consider~$G+W_f$.

\subsection{ (Planarly) Realizable Assignments for a Face}

Throughout this section we consider an embedded graph $G$ together with a
fixed node assignment $A$ and a fixed face $f$ of $G$.  The goal of
this section is to characterize when $A$ is realizable (in a planar
way) for $f$.  We first collect some necessary conditions for a
realizable assignment.  

%\vspace{-1ex}
\begin{condition}[parity] The number of free valencies
  assigned to~$f$ is even.
\end{condition}
%\vspace{-1ex}

Furthermore, we list certain \emph{indicator sets} of vertices
assigned to $f$ that demand additional valencies outside the set to
which they can be matched, as otherwise an augmentation is impossible.
Note that these sets may overlap.

\vspace{1ex}
\begin{compactenum}[(1)]
\item {\bf Joker:} A vertex in $\2V$ whose neighbors are not assigned
  to~$f$ demands \emph{one} valency.\label{prop:deg2}
\item {\bf Pair:} Two adjacent vertices in $\2V$ demand \emph{two}
   valencies. \label{prop:adj_deg2}
\item {\bf Leaf:} A vertex in $\1V$ whose neighbor has degree 3
  demands \emph{two} valencies from two \emph{distinct
    vertices}. \label{prop:deg1}
\item {\bf Branch:} A vertex in $\1V$ and an adjacent vertex in $\2V$
demand \emph{three} valencies from at least
  \emph{two distinct vertices} with at most one valency
  adjacent to the vertex in $\2V$.\label{prop:deg12}
\item {\bf Island:} A vertex in $\0V$ demands \emph{three} valencies
  from \emph{distinct vertices}. \label{prop:deg0}
\item {\bf Stick:} Two adjacent vertices of degree 1 demand
  \emph{four} valencies of which \emph{at most two} belong to the
  \emph{same vertex}. \label{prop:adj_deg1}
\item Two vertices in~$\0V$ demand \emph{four} valencies; {at most
    two} from the \emph{same vertex}.\label{prop:adj_deg0}
\item {\bf 3-cycle:} A cycle of three vertices in $\2V$ demands
  \emph{three} valencies. \label{prop:3cycle}
\end{compactenum}

%\vspace{-1ex}
\begin{condition}[matching]
  The demands of all indicator sets formed by vertices assigned to~$f$
  are satisfied.
\end{condition}
%\vspace{-1ex}

Each indicator set contains at most three vertices and provides at
least the number of valencies it demands; only sets of type
(\ref{prop:adj_deg0}) provide more.  The demand of a joker is
implicitly satisfied by the parity condition.  We call an indicator
set with maximum demand \emph{maximum indicator set}, and we denote
its demand by $\kmax$.  Note that $\kmax \leq 4$.  We observe that
inserting edges does not increase $\kmax$.
%
% ----------------- Observation eingefuegt --------------------------
\begin{observation}
  \label{obs:kmax-increase}
  Inserting an edge $uv$ into~$f$ does not increase~$\kmax$.
\end{observation}

\begin{proof}
  Let $k$ and $k'$ denote $\kmax$ before and after the insertion of $uv$, respectively. We show $k' \leq k$.
  If $k' =4$, then after the insertion there is a stick or an indicator set of type~\ref{prop:adj_deg0}. 
  Since a stick can only be obtained from a set of type~\ref{prop:adj_deg0}, we have       $k = 4$. 
  If $k' = 3$, then after the insertion there is a branch or an island. 
  Since a branch can only be obtained from an island or a stick, we have $k \geq 3$.
  If $k' = 2$, then after the insertion there is a pair or a leaf. 
  Since a pair can only be obtained from two leaves, we have $k \geq 2$.
\end{proof}
% ----------------- Ende Observation ---------------------------------
The following
lemma reveals the special role of maximum indicator sets.

\newcommand{\lemmaMaxIndicator}{Let~$S$ be a maximum indicator set
  in~$f$. Then~$A$ satisfies the matching condition for~$f$ if and
  only if the demand of~$S$ is satisfied.}

\begin{lemma}
  \label{lem:max-indicator}
  \lemmaMaxIndicator
\end{lemma}
% --------------- Proof eingefuegt ----------------
\begin{proof}
  Clearly, if~$A$ satisfies the matching condition than in particular
  the demand of~$S$ is satisfied.  Hence, assume that the demand
  of~$S$ is satisfied.  We prove that for any indicator set~$U$ of
  vertices assigned to~$f$ the demands are satisfied.  Observe that
  the demand of an indicator set that is contained in~$S$ is trivially
  satisfied, we may thus assume that~$U$ contains vertices outside
  of~$S$.  We distinguish cases based on the demand~$\kmax$ of~$S$.

  \textbf{Case I: $\kmax = 4$.}  Then $S$ consists either of a stick
  or a set of type (\ref{prop:adj_deg0}), which is a pair of isolated
  vertices.  Let~$U$ be any indicator set distinct from~$S$.  Assume
  that~$U$ demands four vertices.  If~$U$ is disjoint from~$S$, then
  $S$ provides the demanded valencies.  Otherwise, both~$S$ and~$U$
  consist of a pair of isolated vertices, and they share a common
  vertex.  Since the demand of~$S$ is satisfied, there are at least
  two more assigned valencies provided by vertices outside of~$S \cup
  U$.  Together with~$S \setminus U$, they provide the demanded
  valencies for~$U$.  The same argument applies if~$U$ consists of an
  island, and hence demands three valencies.

  If~$U$ demands three or fewer valencies and it is not an isolated
  vertex, then it is either contained in~$S$ or disjoint from it.  In
  the former case its demand is satisfied, in the latter case the
  demand is satisfied by~$S$ since an island, which is contained in $S$, is the only
  indicator set demanding valencies from three different vertices.

  \textbf{Case II: $\kmax = 3$.} Then~$S$ consists either of a 3-cycle,
   an island, or a branch.  If~$S$ is a 3-cycle, then any
  other indicator set is either completely contained in~$S$ or
  disjoint from it, and it hence provides the necessary valencies
  (even for an isolated vertex).

  If~$S$ consists of an island~$s$, observe that~$\kmax=3$ implies
  that there is no other island assigned to~$f$.  The island~$s$
  provides the necessary valencies for all indicator sets, except for
  a branch or a leaf. Assume that~$U$ is a branch.  Since~$s$ demands
  three valencies from distinct vertices, there is a vertex~$v \notin
  U \cup \{s\}$ assigned to~$f$.  Together~$s$ and~$v$ provide the
  valencies for~$U$. The case that $U$ is a leaf can be treated analogously.

  Finally, consider the case that~$S$ consists of a branch.  If~$U$
  consists of an island~$u$, then there must be a vertex~$v \notin S
  \cup \{u\}$ providing a valency.  Then~$S \cup \{v\}$ provide the
  demanded valencies for~$u$.  If~$U$ is not an island, it demands at
  most three valencies from at most two different vertices.  Hence,
  if~$U$ is disjoint from~$S$, $S$ provides the demanded valencies
  for~$U$.  It remains to deal with the case that~$U$ is a branch
  sharing its degree-2 vertex with~$S$.  But then the situation
  for~$U$ and~$S$ is completely symmetric, and the demands for~$U$ are
  satisfied.

  \textbf{Case III: $\kmax = 2$.} Since the demands of jokers are
  always satisfied due to the parity condition, in this case all
  indicator sets consists either of pairs or of leaves.  If~$S$
  and~$U$ are both leaves, their situation is again completely
  symmetric.  If~$S$ and~$U$ are a leaf and a pair, respectively, they
  mutually satisfy their demands.  It remains to deal with the case
  that~$S$ and~$U$ both consist of pairs.  If~$S$ and~$U$ are
  disjoint, they mutually satisfy their demands.  If they share a
  vertex~$S$ and~$U$ are again completely symmetric.
\end{proof}
% ---------------- ende Proof ------------------------
%
\newcounter{helperlemma}
\setcounter{helperlemma}{\thelemma}
\addtocounter{lemma}{1}

The necessity of the parity and the matching condition is obvious; we
prove that they are also sufficient for a node assignment to be
realizable for $f$.

\newcommand{\theoremCharacMatch}{$A$ is realizable for $f$
  $\Leftrightarrow$ $A$ satisfies the parity and matching condition
  for~$f$.}

\begin{theorem}
  \label{thm:characMatch}
  \theoremCharacMatch
\end{theorem}
The following proof of Theorem~\ref{thm:characMatch} postpones
the case that $A$ assigns less than seven vertices to~$f$ to
Lemma~\ref{lem:matching6}, which handles this by a case distinction.
%
% --------------------------- Proof eingefuegt ----------------------------
\begin{proof}
  If~$A$ assigns less than seven vertices to~$f$, the statement
  follows from Lemma~\ref{lem:matching6}.  Moreover, the parity
  condition and the matching condition are necessary.  In the
  following we assume that~$A$ assigns at least seven vertices to~$f$
  and satisfies the parity condition and the matching condition
  for~$f$.
  Suppose there exists a partial augmentation $W_1$ of $f$ such that
  $A$ still assigns $k \ge 6$ vertices to $f + W_1$ and each assigned
  vertex has degree 2.  We define the graph $H^c$ that consists of the
  vertices assigned to $f + W_1$ and contains an edge between two
  vertices if and only if they are not adjacent in $f + W_1$.  Since
  each assigned vertex in $f+W_1$ has degree 2, it has at most two
  adjacencies in $f+W_1$ and at least $k-1-2 = k-3 \geq k/2$ (for
  $k\geq6$) adjacencies in $H^c$. Thus, by a theorem of
  Dirac~\cite{d-stag-52}, a Hamiltonian cycle exists in $H^c$, which
  induces a perfect matching $W_2$ of the degree-2 vertices in $f
  +W_1$.  Hence $W_1 \cup W_2$ is a 3-regular augmentation for $f$.
  In the remainder of this proof we show that such a partial
  augmentation $W_1$ always exists.

  We begin with the following observation. Let $S$ denote an island or
  a stick and let $e$ denote an edge between two valencies in $f$.
  Splitting $e$ and connecting the resulting half-edges to the vertex,
  respectively the vertices, in $S$ yields an augmentation
  $\{e_1,e_2\}$ such that the vertices in $S$ have degree~2 in $f +
  \{e_1,e_2\}$.  We refer to this procedure as \emph{clipping in} $S$.

  In the following we construct a partial augmentation $W_1$ for all
  possible assignments for~$f$.  In order to identify pairs of
  degree-0 vertices with sticks, in a first step we arbitarily choose
  pairs of degree-0 vertices and connect them by an edge.  Note, that
  this in particular means that there remains at most one island
  assigned to $f$.  Then we distinguish the possible assignments by
  the number of degree-1 vertices that are involved in a leaf or a
  branch.  We denote the set of these vertices by $\1X$.  Since the
  vertices in $\1X$ are mutually non-adjacent, each edge between two
  of these vertices may occur in $W_1$.

  For $|\1X| > 1$ we hence connect the vertices in $\1X$ pairwise and
  clip in possibly existing sticks and islands according to the
  observation above.  If $|\1X|$ is odd there remains one vertex $x
  \in \1X$.  However, the augmentation constructed so far contains at
  least one edge, which we split.  Then we connect the resulting
  half-edges to $x$.  Thus, $x$ becomes a degree-3 vertex and is no
  longer assigned.  Nevertheless, the condition of six assigned
  vertices in $f+W_1$ is still satisfied since there were at least
  seven vertices assigned to $f$.

  If $|\1X| = 1$, let $x$ denote the unique vertex in $\1X$.  As $A$
  satisfies the matching condition for $f$, there is at least one
  vertex $u$ outside the indicator set of $x$ to which we can connect
  $x$.  Thus, $x$ becomes a degree-2 vertex.  If $u$ becomes a
  degree-3 vertex, it is no longer assigned. However, according to the
  same argument as before, this is no problem.
  If $u$ was a vertex in a stick or an island, connecting $x$ to $u$
  yields a new degree-1 vertex in $\1X$ replacing $x$.  In this case,
  we repeat the procedure above until no new degree-1 vertex comes up
  in $\1X$.  The resulting matching contains at least one edge and we
  clip in sticks and islands.

  If $|\1X| = 0$, there exist no leaves and no branches.  The only
  vertices whose degrees need to be increased by $W_1$ are those in
  sticks or islands. All other assigned vertices have degree 2.  Let
  $n'$ denote the number of assigned islands and sticks.  If $n' = 1$,
  there are at least six, respectively seven, further degree-2
  vertices assigned since in total A assigns at least seven vertices
  to $f$ and satisfies the parity condition.  In both cases this
  yields at least eight assigned vertices.  We can hence connect the
  stick or the island with two degree-2 vertices, still having at
  least six assigned vertices in $f + W_1$.  If $n' \geq 2$ we connect
  two arbitrary indicator sets, which are either two sticks or a stick
  and an island, in the obvious way such that each vertex has degree 2
  afterwards.  All further sticks or islands are then clipped in.
  Hence in each case we find a corresponding partial
  augmentation~$W_1$, which concludes the proof.
\end{proof}
% ---------------------------Ende Proof -----------------------------------

\newcounter{oldlemma}
\setcounter{oldlemma}{\thelemma}
\setcounter{lemma}{\value{helperlemma}}
\begin{lemma}
  \label{lem:matching6}
  Let  $A$ assign less that seven vertices to $f$.
  $A$ is realizable for~$f$
  $\Leftrightarrow$~$A$ satisfies the parity and matching condition
  for~$f$.
\end{lemma}
\setcounter{lemma}{\value{oldlemma}}
%
% -------------------- Proof eingefuegt ------------------
\begin{proof}
  Let~$V_f$ denote the vertices assigned to~$f$.  Recall that~$\ii V_f$
  then denotes the number of vertices in~$V_f$ with degree~$i$.  By
  assumption, we have~$|\0V_f| +|\1V_f|+|\2V_f| \le 6$.  We denote
  again the maximum number of valencies demanded by any indicator set
  by $\kmax$, and by assumption $A$ satisfies the parity condition and
  the demands of all indicator sets.  At the beginning we show how to
  solve the following basic situation.  Consider a set $N$ of four
  degree-1 vertices $u_1, u_2, u_3$ and $u_4$, possibly belonging to
  larger indicator sets, and an even number of at most six free
  valencies provided by at most two vertices $v_1$ and $v_2$ not in
  $N$.  Figure~\ref{fig:basicCases} shows the only possible
  occurrences of this situation together with a 3-regular
  augmentation.  We now reduce more complicated situations to these
  cases.

 \begin{figure}[tb]
    \centering
    \includegraphics{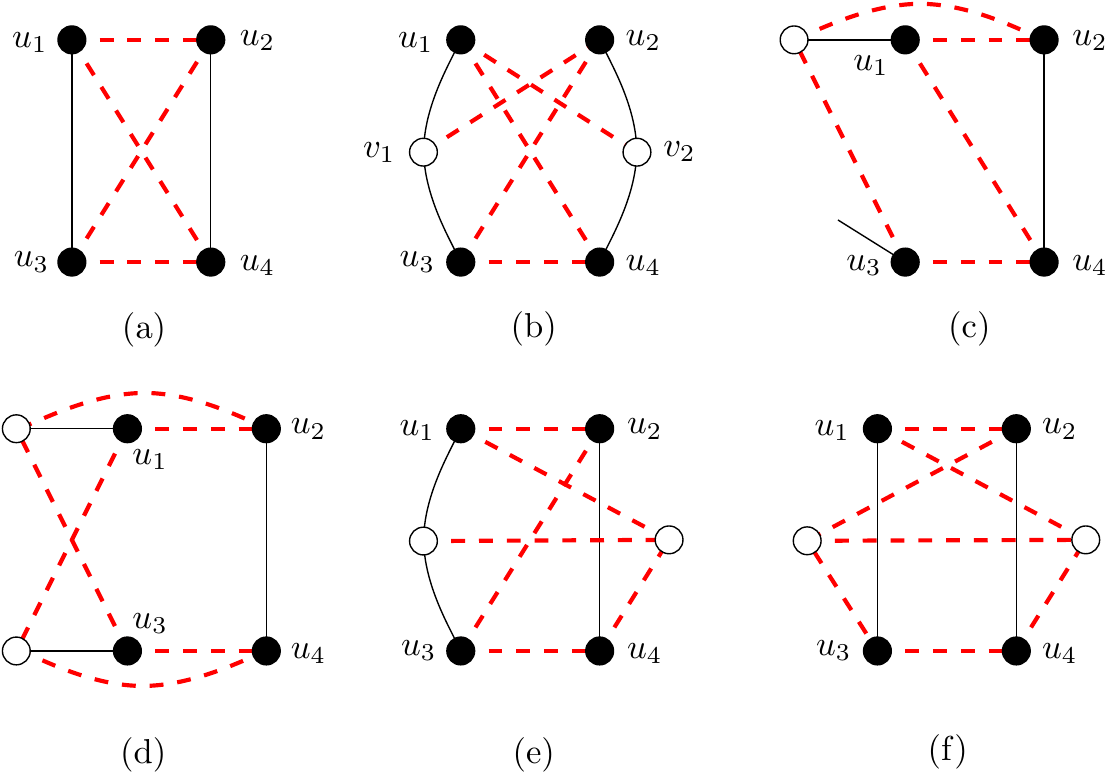}
    \caption{Basic case of Lemma~\ref{lem:matching6}.  The possible
      occurrences of four degree-1 vertices in~$N$ (black disks) and an even number of up to six free valencies provided by at most two further
      vertices~$v_1,v_2$.  For each case a corresponding augmentation
      (dashed edges) is shown.  The black adjacencies are possible but
      not necessarily present, which only yields simpler situations.
      (a) no further valencies. (b) two valencies by two further
      vertices.  This augmentation also applies if $v_1$ and $v_2$ are
      adjacent. (c) two valencies by one vertex. (d) four valencies by
      two degree-1 vertices. (e) four valencies by a degree-0 and a
      degree-2 vertex. Inserting $v_1v_2$ yields (c). (f) six
      valencies by two vertices. Inserting $v_1v_2$ yields (d).}
    \label{fig:basicCases}
  \end{figure}

  \textbf{Case 1: $|\0V_f| \geq 4$.}  Consider two pairs of degree-0
  vertices and connect each pair by an edge.  This yields four
  degree-1 vertices in a set $N$.  Outside $N$ there is an even number
  of at most six free valencies, as two further vertices cannot
  provide more valencies.  Thus, we are done according to the basic
  case above.

  \textbf{Case 2: $2 \leq |\0V_f| \leq 3$.}  Consider a fixed set $N$
  of two degree-0 vertices and connect them by an edge.  Let~$k$ be
  the number of assigned valencies outside of~$N$.  Observe that
  outside~$N$ there are at most four vertices, among them at most one
  degree-0 vertex.  Hence~$k \le 9$.  Conversely, the demand of~$N$ is
  satisfied, and hence~$k \ge 4$, moreover, $k$ is even by the parity
  condition.  If~$k = 4$, then the four remaining valencies in $N$ can
  be arbitrarily matched to those outside $N$ since the vertices in
  $N$ form a connected component by themselves.  

  For the case~$k=6$, we now distinguish cases based on $|\0V_f|$.  If
  $|\0V_f| = 2$ there is no additional degree-0 vertex outside $N$.
  Then any set of at most four vertices providing six valencies
  outside $N$ contains at least two degree-1 vertices.  We add two
  such vertices to~$N$, reducing the valencies outside~$N$ to two, and
  we are done.

  If $|\0V_f| = 3$, the additional degree-0 vertex~$v$ outside~$N$
  already provides three valencies outside.  To reach a sum of six,
  besides~$v$, there must be either a degree-2 and a degree-1 vertex
  or three degree-2 vertices outside~$N$.  In the former case we
  connect the degree-0 vertex and the degree-2 vertex by an edge
  yielding another degree-1 vertex, which we add to $N$ together with
  the remaining degree-1 vertex.  This results in the situation of
  Figure~\ref{fig:basicCases}(a).  This solution
  also applies for the second case by identifying two degree-2
  vertices with the degree-1 vertex in the first
  case.  This concludes the case~$k=6$.

  Finally, if~$k=8$, there are eight free valencies outside the
  initial set $N$, and hence any set of at most four vertices
  providing these valencies contains at least two degree-1 vertices,
  independent of whether $|\0V_f| = 2$ or $|\0V_f| = 3$.  Adding two
  such degree-1 vertices to $N$ reduces the number of valencies
  outside $N$ to four, and we are done.

  \begin{figure}[tb]
    \centering
    \includegraphics{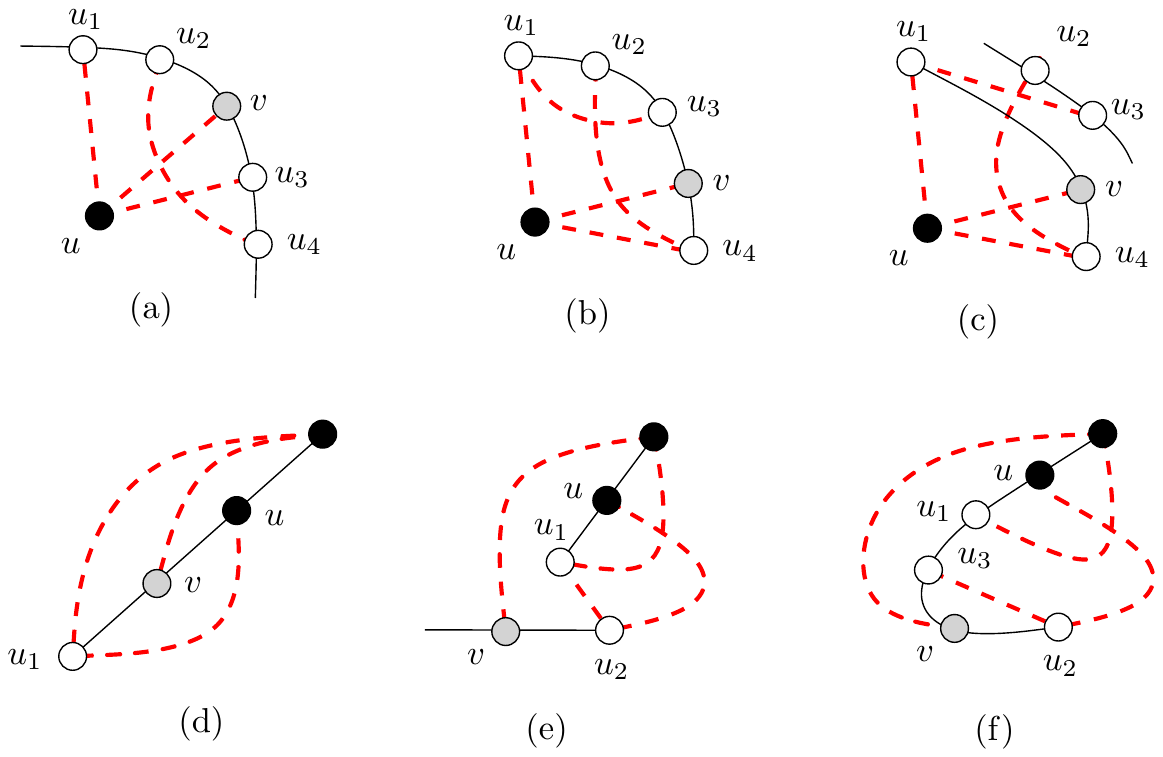}
    \caption{Illustration of the proof of Lemma~\ref{lem:matching6};
      augmentation edges are dashed.  The black vertices form a maximum indicator set~$S$ with demand~$\kmax$, the grey vertex~$v$ has degree~2 and exists due to the demand of~$S$.  All augmentations also apply if
      the solid adjacencies are (partly) dropped.  In (a)-(c)
      $\kmax=3$, and there is a single vertex~$u$ of degree-0. (a)
      $|\1V_f| = 0$, also applies if~$u_1$ and~$u_4$ are adjacent or
      ignored.  Moreover~$\{u_1,u_2\}$ or~$\{u3,u_4\}$ may be
      considered as a single degree-1 vertex.  (b) + (c) $|\1V_f| = 2$
      and $|\2V_f| > 0$; the augmentation also applies if
      $\{u_2,u_3\}$ is considered as a single degree-1 vertex. (d)-(f)
      $\kmax = 3$ and a branch $S$ with degree-2 vertex~$u \in S$.
      (d) One degree-1 vertex besides $S$ not adjacent to $u$; also
      applies if $u_1$ is considered as two degree-2 vertices.  (e)
      Two degree-1 vertices besides $S$, one adjacent to $u$; also
      applies if $u_2$ is considered as two degree-2 vertices.  (f)
      One degree-1 vertex and further degree-2 vertices besides $S$,
      adjacent to $u$.}
    \label{fig:Cases3}
  \end{figure}

  \textbf{Case 3: $|\0V_f| = 1$.}  Let $u$ denote the only degree-0
  vertex, which demands three further valencies from distinct
  vertices, that is, $\kmax \in \{3,4\}$, and there are at least three
  vertices assigned besides $u$ of which at least one vertex $v$ is of
  degree 2 in order to satisfy the parity condition.  More precisely,
  there is an even positive number of valencies provided by two, three
  or four vertices besides $u$ and $v$, as the total number of
  vertices assigned to $f$ is at most six.  We distinguish cases based
  on the demand~$\kmax$.

  If $\kmax = 3$, consider the number of assigned degree-1 vertices.
  Recall that the degree-1 vertices are pairwise non-adjacent, $\kmax$
  would be~4, otherwise.  If~$|\1V_f| = 0$, the demand of $u$ implies
  that~$V \setminus \{u,v\}$ consists of at least two degree-2
  vertices.  It follows from the parity condition that the number of
  these degree-2 vertices is either two or four.
  Figure~\ref{fig:Cases3}(a) shows a solution for four further
  degree-2 vertices.  The case of only two further degree-2 vertices
  can be deduces from Figure~\ref{fig:Cases3}(a) by ignoring $u_2$ and
  $u_4$.
  If~$|\1V_f| = 1$, the demand of~$u$ and the parity condition imply
  that there are exactly two degree-2 vertices in~$V_f \setminus
  \{u,v\}$.  An augmentation is given by Figure~\ref{fig:Cases3}(a)
  identifying $u_1$ and $u_2$ with the degree-1 vertex.
  If~$|\1V_f| = 2$, we have~$|\2V_f \setminus \{v\}| \le 2$ and even.
  This situation is solved by Figure~\ref{fig:Cases3}(b) and (c). If
  there are no additional degree-2 vertices an augmentation results
  from Figure~\ref{fig:Cases3}(a) by identifying $u_1$ and $u_2$ as
  well as $u_3$ and $u_4$ with the two degree-1 vertices (which are
  non-adjacent).
  If~$|\1V_f| = 3$, the parity condition prohibits a further degree-2
  vertex.  Thus, Figure~\ref{fig:Cases3}(c) provides an augmentation
  by identifying $u_2$ and $u_3$ with a degree-1 vertex not adjacent
  to $u_1$ and $u_4$.  The case~$|\1V_f| = 4$ is shown by
  Figure~\ref{fig:basicCases}(e), ignoring the adjacency of $u_2$ and
  $u_4$.

  If $\kmax = 4$, there must exist at least one stick~$S$ that demands
  four additional valencies, of which at most two belong to the same
  vertex.  This is, besides $S$, $u$ and $v$, there is an even
  positive number of valencies provided by at most two further
  vertices.  If~$|\1V_f| \ge 4$, we are in the situation of
  Figure~\ref{fig:basicCases}(e).  If $|\1V_f| = 3$, the parity
  condition prohibits a further degree-2 vertex.  This situation can
  be deduced from of Figure~\ref{fig:Cases3}(c) assuming $u_1$ and
  $u_4$ are adjacent, forming $S$, and identifying $u_2$ and $u_3$
  with the degree-1 vertex; the vertex $v$ may be located arbitrarily.
  If $|\1V_f| = 2$, $S$ contains the only degree-1 vertices, and there
  must be two further degree-2 vertices besides $S$, $u$ and
  $v$.  Thus, Figure~\ref{fig:Cases3}(c) provides a solution assuming
  $u_1$ and $u_4$ form $S$ and the remaining degree-2 vertices are
  located arbitrarily.

 \begin{figure}[tb]
    \centering
    \includegraphics{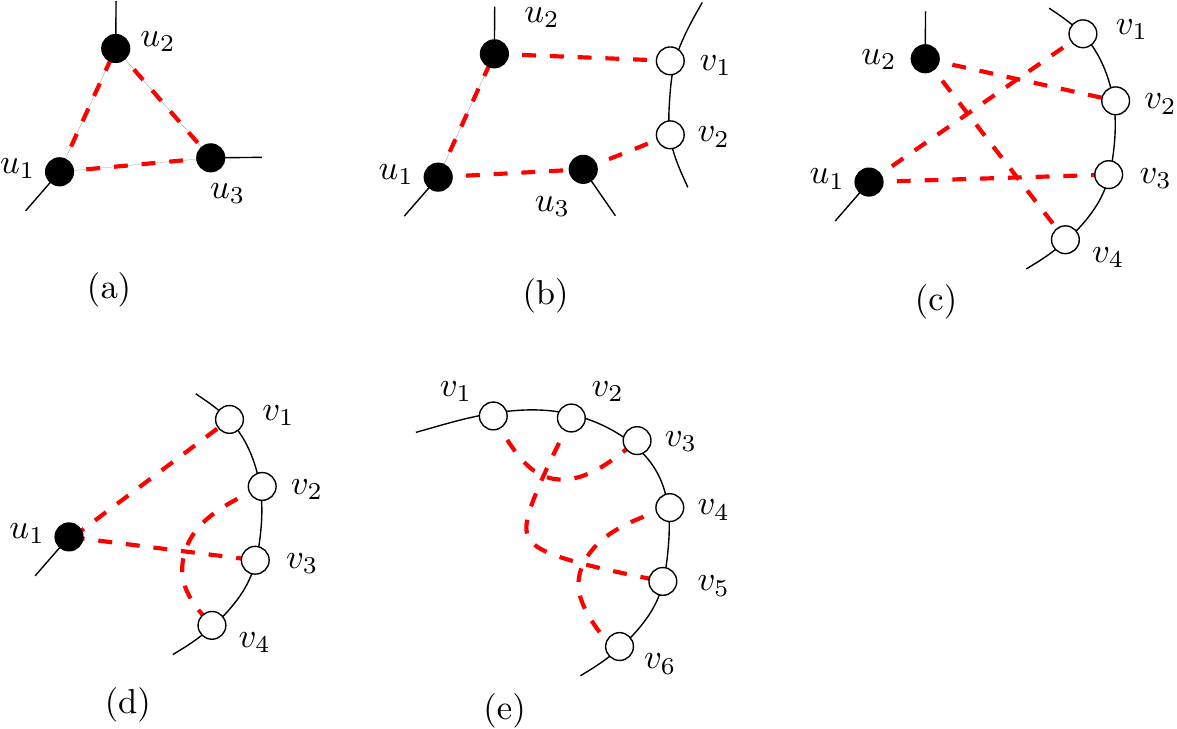}
    \caption{Augmentation (dashed edges) of assignments with
      $\kmax=2$.  The black vertices have degree~1.  All augmentations 
      also apply if the solid adjacencies are (partly) dropped.  (a) $|\1V_f| = 1$;
      augmentation also applies if $u_3$ is considered as two degree-2
      vertices. (b) $|\1V_f| = 3$ and~$|\2V_f| = 2$.  (c) $|\1V_f| =
      2$ and~$|\2V_f| = 4$; also applies if $v_1$ and $v_4$ are
      adjacent or $u_2$, $v_2$ and $v_4$ are ignored.  (d) $|\1V_f| =
      1$ and~$|\2V_f| = 4$; also applies if $v_1$ and $v_4$ are
      adjacent.  (e) $|\2V_f| = 6$; also applies if $v_1$ and $v_6$
      are adjacent or $v_4$ and $v_6$ are ignored.}
    \label{fig:Cases4}
  \end{figure}

  \textbf{Case 4: $|\0V_f| = 0$.}
  If $\kmax = 1$, there is an even number of jokers assigned to~$f$.
  We connect pairs arbitrarily.

  If $\kmax = 2$ and~$|\1V_f| \ge 4$, we are done according to
  Figure~\ref{fig:basicCases}(a)-(d).  Note that due to $\kmax =2$
  each degree-1 vertex is neither part of a stick, nor of a branch,
  and hence is not adjacent to other assigned vertices.  For~$|\1V_f|
  = 3$, Fig.~\ref{fig:Cases4}(a) and (b) show an augmentation,
  depending on whether the degree-1 vertices are accompanied by two
  degree-2 vertices or not.  Note that the total limit of six vertices
  and the parity condition does not allow for a different number of
  degree-2 vertices.  If~$\1|V_f| = 2$, there are either two or four
  degree-2 vertices, as otherwise the demands or the parity condition
  would be violated.  Corresponding solutions are given by
  Figure~\ref{fig:Cases4}(a) (by identifying $u_3$ with two
  arbitrarily located degree-2 vertices) and
  Figure~\ref{fig:Cases4}(c), respectively.  If~$|\1V_f | = 1$, there
  are two or four degree-2 vertices.  Corresponding augmentations are
  shown in Figure~\ref{fig:Cases4}(c) (ignoring $u_2$, $v_2$ and
  $v_4$) and Figure~\ref{fig:Cases4}(d), respectively.  Finally,
  if~$|\1V_f| = 0$, all valencies assigned to $f$ are provided by
  degree-2 vertices.  Figure~\ref{fig:Cases4}(e) shows an augmentation
  for six valencies.  Ignoring $v_4$ and $v_6$ in
  Figure~\ref{fig:Cases4}(e) yields an augmentation of four valencies.
  Augmenting two non-adjacent degree-2 vertices is trivial.

  If $\kmax = 3$, we distinguish cases based on whether there is a
  3-cycle assigned to~$f$.  If this is the case, let~$C$ the vertex
  set of such a cycle.  Due to the parity condition there must be a
  further vertex~$v$ of degree 2 assigned to $f$.  More precisely, by
  the demand of~$C$ and the parity condition, there is an even
  positive number of valencies provided by one or two further vertices
  outside~$C \cup \{v\}$.  If there are two valencies provided
  outside~$C \cup \{v\}$, we pair the valencies of~$C$ arbitrarily
  with the remaining three valencies.  If there are four valencies
  assigned, then they are provided by two non-adjacent degree-1
  vertices.  Connecting them by an edge reduces to the previous case.
  Now assume that there is no 3-cycle, and hence there is a branch~$S$
  demanding three valencies.  Due to the parity condition there is a
  further degree-2 vertex $v$ assigned to $f$.  By the parity
  condition there is an even positive number of valencies outside~$S
  \cup \{v\}$, which is provided by one, two or three further
  vertices.  We distinguish cases based on the number~$k$ of these
  valencies.  Note that~$k \le 6$.
  
  If~$k=2$, the valencies are provided by one or two vertices.  We
  note that if it is a degree-1 vertex, then it is not adjacent
  to~$S$, as this would contradict the demand of~$S$, and
  Figure~\ref{fig:Cases3}(d) shows a solution for this case.  The same
  augmentation works if the valencies are provided by two degree-2
  vertices by considering two degree-2 vertices that are not adjacent
  to~$S$ as the single degree-1 vertex~$u_1$ in the figure.  If~$k=4$,
  consider the case that the four valencies outside~$S \cup \{v\}$ are
  provided by two degree-1 vertices.  Figure~\ref{fig:Cases3}(e) shows
  a solution, and this also holds if~$u_2$ is replaced by two degree-2
  vertices that are located arbitrarily.  It remains to deal with the
  case where four valencies are provided by a degree-1 vertex and two
  degree-2 vertices, but the degree-1 vertex is adjacent to~$S$.
  Figure~\ref{fig:Cases3}(e) shows a solution for this case.
  If~$k=6$, all further vertices have degree~1, resulting in~$|\1V_f|
  = 4$, which can be handled by the basic case.  This concludes the
  case~$\kmax=3$.

  If $\kmax = 4$, let $S \subseteq \1V_f$ denote a stick.  If~$|\1V_f|
  = 4$, we are done according to Figure~\ref{fig:basicCases}(a)-(d).
  If~$|\1V_f| = 3$, then we have~$|\2V_f| = 2$ by the demand of~$S$
  and the total limit of six vertices.  In this case outside~$S$ there
  are exactly the four valencies demanded by~$S$, and we may connect
  them arbitrarily since~$S$ forms a connected component.  Finally,
  if~$|\1V_f| = 2$, the demand of~$S$ must be satisfied by four
  degree-2 vertices, and we connect~$S$ to these vertices arbitrarily.
\end{proof}
% ----------------- ende Proof ------------------------------

Given a node assignment $A$ that satisfies the parity and the matching
condition for a face $f$, the following rule picks an edge that can be
inserted into $f$.  Lemma~\ref{lem:R1} states that afterwards the remaining
assignment still satisfies the parity and the matching condition.
Iteratively applying Rule~$\ref{r:1}$ hence yields a (not necessarily planar) realization.

\begin{selectionRule}\label{r:1}
  \begin{compactenum}
  \item If $\kmax \geq 3$ let $S$ denote a maximum indicator
    set. Choose a vertex $u$ of lowest degree in $S$ and connect this
    to an arbitrary assigned vertex~$v \notin S$.
  \item If $\kmax = 2$ and $u$ is a leaf, choose~$S =
    \{u\}$, and connect $u$ to an assigned vertex~$v$.
  \item If $\kmax = 2$ and there is no leaf, let $S$ denote a path
    $xuy$ of assigned vertices in $\2V$.  Connect~$u$ to an arbitrary
    assigned vertex $v\notin S$.
  \item If $\kmax = 2$ and there is neither a leaf nor a path of three assigned
    vertices in $\2V$, let $S$ denote a pair $uw$. Connect $u$ to an
    arbitrary assigned vertex~$v\notin S$.
  \item If $\kmax = 1$, choose~$S=\{u\}$, where $u$ is a
  joker, and connect~$u$ to another joker $v$.
  \end{compactenum}
\end{selectionRule}

\newcommand{\lemmaRuleOne}{Assume $A$ satisfies the parity and matching
  condition for $f$ and let $e$ denote an edge chosen according to
  Rule 1.  Then $A$ satisfies the same conditions for $f + e$.}

\begin{lemma}
  \label{lem:R1}
  \lemmaRuleOne
\end{lemma}
% ------------------------ Proof eingefuegt ------------------------
\begin{proof}
  It follows from Theorem~\ref{thm:characMatch} that $A$ is realizable
  for $f$.  Thus, let $W_f$ denote a 3-regular augmentation for~$f$.
  If $e = uv \in W_f$, we are done since $W_f \setminus \{e\}$ is a
  3-regular augmentation for $f+e$.

  Hence, assume $uv \notin W_f$.  We consider the set $\mW \subseteq
  W_f$ of edges that are incident to vertices in $S$, where $S$ is the
  set determined by the rule.  It is $u \in S$ and $v\notin S$ for all
  rules.  The deletion of $\mW$ in $f + W_f$ yields a set $\sX \cup S$
  of vertices which have again free valencies in $f + (W_f - \mW)$.
  If $v$ is connected to a vertex of~$S$ in $W_f$, then we already
  have $v\in \sX$.  Otherwise, we add an arbitrary edge $vx \in W_f$
  to $\mW$, which yields $v,x \in \sX$.  Clearly, $A$ satisfies the
  parity condition and the matching condition for $f + (W_f - \mW)$,
  and, after insertion of~$e=uv$, at least the parity condition for $f
  + (W_f - \mW) + uv$.  In the following we show that $A$ also
  satisfies the demand of a maximum indicator set $S'$ in $f + (W_f -
  \mW) + uv$.  Then $A$ satisfies the matching condition for $f + (W_f
  - \mW) + uv$ by Lemma~\ref{lem:max-indicator}, and there exists an
  augmentation $\pW$ for $f + (W_f - \mW) + uv$ such that $(W_f - \mW)
  + \pW$ is an augmentation for $f+e$.  Hence, $A$ satisfies the
  parity condition and the matching condition for $f+e$ as claimed by
  the lemma.

  For each subrule we distinguish cases based on the type
  of~$S$.  Recall that the insertion of edges never increases
  $\kmax$.  Thus, any maximum indicator set in $f + (W_f - \mW) + uv$
  demands at most as many valencies as a maximum indicator set in
  $f$.  Note further, that $v$ has degree~3 in $f + (W_f - \mW) + uv$
  if $v$ is not matched to $S$ in $W_f$.

  {\bf Subrule 1:} In this case,~$S$ is a maximum indicator set, and
  we distinguish further cases based on the exact type of~$S$.

  \emph{Case I:} Assume $S$ is a stick or an indicator set of
  type~(\ref{prop:adj_deg0}).  If $v$ was connected to $s \in S$ in
  $W_f$, then, except for the valency at~$s$, which was connected
  to~$u$, the valencies in $\sX \setminus (S \cup \{v\})$ were matched
  to $S$.  Due to the symmetry of $S$, these valencies can be also
  matched to $S$ preserving the necessary valency for $v$ at $u$.
  Thus, $A$ satisfies the parity condition and the matching condition
  for $f + e$, according to Theorem~\ref{thm:characMatch}.

  If $v$ was not connected to $S$ in $W_f$ and $S$ is a set of type
  (\ref{prop:adj_deg0}), we connect the isolated degree-0 vertices
  in~$S$ by an edge, such that $S$ becomes a stick.  Since this does
  not change the demand of $S$, which remains a maximum indicator set,
  Lemma~\ref{lem:max-indicator} and Theorem~\ref{thm:characMatch}
  imply that $A$ is still realizable for $f + (W_f - \mW)$ after this
  insertion.  Thus, we identify this case with the next case, where
  $S$ is a stick.

  Assume that~$S$ is a stick and $v$ was not connected to $S$
  in~$W_f$.  Then $S$ provides three valencies at two distinct
  vertices in $f + (W_f - \mW) + uv$.  The vertices with free
  valencies in~$f + (W_f - \mW) + uv$ are partitioned into two
  disjoint, non-adjacent groups, namely $S$ and~$\sX\setminus (S\cup
  \{v\})$.  We show that no maximum indicator set~$S'$ contains
  vertices of both sets, and then argue that the group that has empty
  intersection with a maximum indicator set~$S'$ provides enough
  valencies to satisfy the demand of~$S'$.

  First observe that~$S$ and~$\sX\setminus (S\cup \{v\})$ are disjoint by
  definition and non-adjacent since~$S$ was a stick (before connecting
  it to~$v$).  Let~$S'$ be any maximum indicator set $f + (W_f - \mW)
  + uv$.  By definition~$S$ does not contain a degree-0 vertex, which
  could belong to an indicator set of type~(\ref{prop:adj_deg0}).
  Since all other indicator sets are connected it follows that
  either~$S' \subseteq S$ or~$S' \cap S = \emptyset$.  
   The set $\sX\setminus (S\cup \{v\})$ provides (at least) $4 + 1 = 5$
  valencies on at least three distinct vertices in $f + (W_f - \mW) +
  uv$.  Recall that if $S$ has originally been a set of type
  (\ref{prop:adj_deg0}), it may have induced six valencies provided by
  $\sX\setminus (S\cup \{v\})$ in $f + (W_f - \mW) +
  uv$.  The valencies in $\sX\setminus (S\cup
  \{v\})$ clearly satisfy the demand of any maximum indicator set
  $S'\subseteq S$.  If $S'\subseteq \sX\setminus (S\cup \{v\})$, then
  $S'$ either consists of at most two vertices or it is a 3-cycle.  In
  both cases there exists at least one valency in $\sX\setminus (S\cup
  \{v\})$ outside $S'$ since $S'$ provides at most three valencies.
  Thus, $S$ together with this valency provides four valencies on at
  least three distinct vertices, which satisfies the demand of any
  indicator set $S' \subseteq \sX\setminus (S\cup \{v\})$.

  \emph{Case II:} Assume $S$ is a 3-cycle or an island.  If $v$ was
  connected to $s \in S$ in $W_f$, we are done by the same symmetry
  argument as in the beginning of Case I.

  If $v$ was not connected to $S$ in $W_f$, $S$ provides two valencies
  at at most two distinct vertices in $f + (W_f - \mW) + uv$, and
  again each maximum indicator set $S'$ is completely contained either
  in $S$ or in $\sX\setminus (S\cup \{v\})$. 

   The latter provides $3 +
  1 = 4$ valencies on at least two distinct vertices in $f + (W_f - \mW) + uv$.  
  This clearly satisfies the demand of any maximum indicator set
  $S'\subseteq S$, 
  since $S'$ demands at most two valencies in $f + (W_f - \mW) + uv$.  
  If $S'\subseteq \sX\setminus
  (S\cup \{v\})$, there exists at least one valency in $\sX\setminus
  (S\cup \{v\})$ outside $S'$ since each maximum indicator set (for
  $\kmax = 3$) provides only three valencies.  Thus, $S$ together with
  this valency provides three valencies at three distinct vertices if
  $S$ was a 3-cycle (before connecting~$u$ and~$v$), and at two
  distinct vertices if $S$ was an island.  In both cases this satisfies
  the demand of any maximum indicator set $S' \subseteq \sX\setminus
  (S\cup \{v\})$ with $\kmax = 3$.  Note that in the latter case
  $S'\subseteq \sX\setminus (S\cup \{v\})$ is no island since together
  with~$u$ this would have induced a set of type~(\ref{prop:adj_deg0})
  in $f$, contradicting $\kmax = 3$.

  \emph{Case III:} Assume $S$ is a branch.  Recall that $u$ is the
  degree-1 vertex in $S$.  Denote the degree-2 vertex in $S$ by $r$.
  The second vertex~$x$ besides $u$ that is adjacent to $r$ in $f +
  (W_f - \mW) + uv$ has at least degree 2 since $S$ can be connected
  to this vertex by at most one edge in $W_f$.  However, unlike the
  previous cases, it is now possible that a maximum indicator set~$S'$
  has nonempty intersection with both $S$ and $\sX\setminus (S\cup
  \{v\})$.  This is the case where~$S' = \{r,x\}$ and $x$ has
  degree~2, and we have to consider this case in addition to the usual
  ones.  Observe that $S$ provides two valencies at $u$ and $r$ in $f
  + (W_f - \mW) + uv$.

  If $v$ was connected to $r \in S$ in $W_f$, $v$ has degree~3 in $f +
  (W_f - \mW) + uv$, as $uv \notin W_f$, and $\sX\setminus (S\cup
  \{v\})$ provides exactly two valencies at two degree-2 vertices (the
  ones that were adjacent to~$u$ in~$W_f$.  This clearly satisfies the
  demand of each maximum indicator set $S'\subseteq S$ since in this
  case $S'$ demands at most two vertices.  If $S' \subseteq
  \sX\setminus (S\cup \{v\})$, then since $S' \subseteq \sX\setminus
  (S\cup \{v\})$ consists of two degree-2 vertices, $S'$ demands two
  valencies, which are satisfied by $S$.  If $S' = \{x,r\}$ is the
  pair containing vertices of both groups, its demand is satisfied by
  $u$ and the degree-2 vertex in $\sX\setminus (S\cup \{v\})$
  different from~$x$.

  If $v$ was not connected to $r \in S$ in $W_f$, $\sX\setminus (S\cup
  \{v\})$ provides $3+1 = 4$ valencies on at least three vertices, due
  to $\kmax = 3$.  This clearly satisfies the demand of each maximum
  indicator set $S'\subseteq S$.  If $S'\subseteq \sX\setminus (S\cup
  \{v\})$, there exists at least one valency in $\sX\setminus (S\cup
  \{v\})$ outside $S'$ since each maximum indicator set provides only
  three valencies.  Thus, $S$ together with this valency provides
  three valencies at three distinct vertices, which satisfies the
  demand of each maximum set $S' \subseteq \sX\setminus (S\cup \{v\})$
  with $\kmax = 3$.  Finally, if~$S' = \{r,x\}$ is the pair containing
  vertices of both groups, the demand of $S'$, which is~2, is easily
  satisfied by the remaining three valencies in $\sX\setminus (S\cup
  \{v\}) $.  This concludes the treatment of subrule~1.

  {\bf Subrules 2 and 4:} In this case the set $S$ is a leaf or a pair
  (of adjacent degree-2 vertices), and we have~$\kmax=2$.  If $v$ was
  matched to $s \in S$ in $W_f$, we are done by the same symmetry
  argument as before.  

  If $v$ was not matched to $S$ in $W_f$, then $S$ provides one
  valency at one vertex in $f + (W_f - \mW) + uv$.  Again we argue
  that each maximum indicator set $S'$ is either in $S$ or in
  $\sX\setminus (S\cup \{v\})$.  If~$S'$ is a leaf, this is clear and
  if it is a pair $\{u,w\}$, that is if subrule 4 is applied, the
  neighbors of $u$ and $w$ have degree~$3$ as subrule~3 would have
  been applied otherwise.  Hence~$S$ and~$\{u,w\}$ are either equal or
  disjoint.

  The set $\sX\setminus (S\cup \{v\})$ provides $2+1=3$ valencies on
  at least two distinct vertices.  This clearly satisfies the demand
  of each maximum indicator set $S'\subseteq S$ since in this case
  $S'$ demands only one valency.  If $S'\subseteq \sX\setminus (S\cup
  \{v\})$, there exists at least one valency in $\sX\setminus (S\cup
  \{v\})$ outside $S'$ since each maximum indicator set (for $\kmax =
  2$) provides only two valencies.  Thus, $S$ together with this
  valency provides two valencies on at least two distinct vertices,
  which satisfies the demand of each maximum set $S' \subseteq
  \sX\setminus (S\cup \{v\})$ with $\kmax = 2$.

  {\bf Subrule 5:} Since $S$ is a joker, that is $\kmax = 1$, all
  vertices in $\sX$ are jokers in $f + (W_f - \mW) $.  Jokers can be
  matched arbitrarily, and thus, there exists an augmentation that
  contains $e = uv$.  By Theorem~\ref{thm:characMatch}, $A$ then
  satisfies the parity condition and the matching condition for $f+e$.

  {\bf Subrule 3:} Subrule 3 differs from the remaining rules since
  $S$ is no indicator set.  Consequently, $W_f$ may contain the edge
  $xy$ between the endpoints of the path forming~$S$.  If~$W_f$ does
  not contain~$xy$, then $\sX\setminus (S\cup\{v\})$ provides four
  valencies in $f + (W_f - \mW) +uv$ (three for~$S$ and one for~$v$).
  If~$xy \in W_f$, then $S\cup\{v\}$ provides four
  valencies in $f + (W_f - \mW) +uv$.  Thus, in total $\sX$ provides
  at least four valencies in $f + (W_f - \mW) +uv$.  It hence
  satisfies the demand of any maximum indicator set $S'$ since $\kmax
  = 2$ and $S'$ is no leaf, otherwise subrule~2 would have been
  invoked instead of subrule~3.
\end{proof}
% ---------------------- Ende Proof ---------------------------

Our next goal is to extend this characterization and the construction
of the assignment to the planar case.  Consider a path of degree-2
vertices that are incident to two distinct faces~$f$ and~$f'$ but are
all assigned to~$f$.  Then a \emph{planar} realization for~$f$ may not
connect any two vertices of the path.  Hence the following sets of
vertices demand additional valencies, which gives a new condition.
\begin{compactenum}[(1)]\itemsep=0ex
\item A path $\pi$ of $k > 2$ assigned degree-2 vertices that are
  incident to two distinct faces (end vertices not adjacent) demands
  either $k$ further valencies or at least one valency from a
  different connected component.  \label{prop:path}
\item A cycle $\pi$ of $k > 3$ assigned degree-2 vertices that are
  incident to two distinct faces demands either $k$ further valencies
  or at least two valencies from two distinct connected components
  different from $\pi$.  \label{prop:cycle}
\end{compactenum}

%\vspace{-1ex}
\begin{condition}[planarity]
  The demand of each path of $k > 2$ and each cycle of $k > 3$
  degree-2 vertices that are incident to two faces and that are assigned to $f$,
  is satisfied.
\end{condition}
%\vspace{-1ex}
%
%
Obviously, the planarity condition is satisfied if and only if the
demand of a longest such path or cycle is satisfied.  We prove for a
node assignment~$A$ and a face~$f$ that the parity, matching, and
planarity condition together are necessary and sufficient for the
existence of a planar realization for a face~$f$.  To construct a
corresponding realization we give a refined selection rule that
iteratively chooses edges that can be embedded in~$f$, such that the
resulting augmentation is a planar realization of~$A$ for~$f$.  The
new rule considers the demands of both maximum paths and cycles and
maximum indicator sets, and at each moment picks a set with highest
demand.  If an indicator set is chosen, essentially Rule~\ref{r:1} is
applied.  However, we exploit the freedom to choose the endpoint~$v$
of~$e=uv$ arbitrarily, and choose $v$ either from a different
connected component incident to~$f$ (if possible) or by a right-first
(or left-first) search along the boundary of~$f$.  This guarantees
that even if inserting the edge~$uv$ splits~$f$ into two faces~$f_1$
and~$f_2$, one of them is incident to all vertices that are assigned
to~$f$.  Slightly overloading notation, we denote this face by~$f+e$
and consider all remaining valencies assigned to it.  We show in
Lemma~\ref{lem:R2} that $A$ then satisfies all three conditions
for~$f+e$ again.

\begin{selectionRule}\label{r:2}
  Phase 1: Different connected components assign valencies to~$f$.
  \begin{compactenum}
  \item If there exists a path (or cycle) of more than $\kmax$
    assigned degree-2 vertices, let $u$ denote the middle vertex
    $v_{\lceil k/2 \rceil}$ of the longest such path (or cycle) $\pi =
    v_1,\dots v_k$. Connect $u$ to an arbitrary assigned vertex $v$ in another
    component.
  \item If all paths (or cycles) of assigned degree-2 vertices have
    length at most $\kmax$, apply Rule~\ref{r:1}, choosing the vertex
    $v$ in another component.
  \end{compactenum}
  Phase 2: All assigned valencies are on the same connected component.
  Consider only paths of assigned degree-2 vertices that are incident
  to two distinct faces:
  \begin{compactenum}
  \item If there exists a path that is longer than $\kmax$, let $u$
    denote the right endvertex~$v_k$ of the longest path $\pi =
    v_1,\dots v_k$.  Choose~$v$ as the first assigned vertex found by
    a right-first search along the boundary of~$f$, starting from~$u$.
  \item If all paths have length at most $\kmax$, apply
    Rule~\ref{r:1}, choosing $v$ as follows:\\ Let~$v_1,v_2$ denote the
    first assigned vertices not adjacent to~$u$ found by a left- and
    right-first search along the boundary of~$f$, starting at~$u$.  If
    $S$ is a branch and one of~$v_1,v_2$ has degree~2, choose it
    as~$v$.  In all other cases choose~$v = v_1$.
  \end{compactenum}
\end{selectionRule}

\newcommand{\lemmaRuleTwo}{Assume $A$ satisfies the parity, matching,
  and planarity condition for~$f$ and let $e$ be an edge chosen
  according to Rule~\ref{r:2}.  Then~$A$ satisfies all conditions also
  for~$f+e$.}

\begin{lemma}
  \label{lem:R2}
  \lemmaRuleTwo
\end{lemma}
%
% ------------------------ Proof eingefuegt ----------------------
\begin{proof}
  Let $A$ be a node assignment that satisfies the parity condition and
  the matching condition for $f$. Let $A$ further satisfies the
  planarity condition for $f$.  Suppose $e = uv$ is chosen by one of
  the subrules of Rule~\ref{r:2}.  If $e$ is chosen by one of the
  second subrules of Rule~\ref{r:2}, which refer to Rule~\ref{r:1},
  let $S$ denote the set of vertices defined by Rule~\ref{r:1} in
  order to determine $u$.  Otherwise, define $S$ as the longest path
  (or cycle) from which the rule choses $u$.  In both cases it is $u
  \in S$ and $v\notin S$.  If $S$ originates from Rule~\ref{r:1}, we
  define $k:= \kmax$. Otherwise, $k$ denotes the number of valencies
  provided by $S$. Note that in the former case $k$ also describes the
  number of valencies provided by $S$, unless $S$ is an indicator set
  of type (\ref{prop:adj_deg0}).

  Let $\pi$ denote a maximum path (or cycle) in $f+e$. We denote the
  length of $\pi$ by $|\pi|$. Recall that a maximum path consists of
  $|\pi|\geq 3$ assigned degree-2 vertices that are incident to two
  distinct faces.  We distinguish two cases. The first case considers
  a maximum path (or cycle) $\pi$ in $f+e$ that already exists in
  $f$. This is, $\pi$ does not contain $e$. We write $e\not \subseteq
  \pi$. In $f$ $\pi$ is not necessarily maximum.  The second case
  considers a maximum path (or cycle) $\pi$ in $f+e$ that contains
  $e$, i.e., that arises due to the insertion of $e$.  In the
  following we prove that the demand of $\pi$ is satisfied in $f+e$ in
  both cases.  Since $\pi$ is maximum, this implies that the demand of
  all paths (or cycles) considered in the planarity condition is
  satisfied. Thus, the planarity condition is satisfied for $f+e$.
  The parity condition is obviously satisfied in $f+e$, since the
  number of assigned valencies decreases by 2 due to the insertion of
  $e$. In a final step we will prove that the matching condition is
  satisfied for $f+e$.  We first focus on the \emph{planarity
    condition}.

  {\bf Case I:} $e\not \subseteq \pi$.  Then, it is $|\pi|\leq
  k$. Otherwise the rule would have chosen $\pi$ for $S$ in $f$.
  Furthermore, $\pi$ does not intersect with $S$, i.e., $\pi$ is
  either contained in $S$ or $S$ and $\pi$ are disjoint. If $\pi$
  intersected with $S$, $S$ would contain at least one degree-2 vertex
  that is incident to two distinct faces. If $S$ originates from
  Rule~\ref{r:1}, $S$ is an indicator set, and the only indicator sets
  that contain such a vertex are pairs and jokers.  These, however,
  only occur for $\kmax \leq 2$. Thus, in this case $\pi$ with $|\pi|
  \geq 3$ does not exist.  If $S$ is defined by the first subrules of
  Rule~\ref{r:2}, it is either a maximum path (or cycle) or it
  contains no degree-2 vertex incident to two faces. The latter may happen if subrule~1
  of Phase~1 chooses an ``inner'' path (one that is not incident to
  two distinct faces) for $S$.  In the former case the assumption that $\pi$ intersects $S$ yields a longer path (or cycle) in $f$
  contradicting the maximality of $S$.
  We distinguish cases based on the position of $\pi$ in $f+e$.

  \emph{Case A:} $\pi \cap S = \emptyset$.  The set $S$ provides at
  least $k-1$ valencies in $f+e$. Recall that $|\pi| \leq k$.  If
  $|\pi| < k$ the demand of $\pi$ is satisfied by $S$.  If $|\pi| = k$
  there are at least $k + (k - 1)$ valencies assigned to $f +
  e$. Since the parity condition is obviously preserved by the
  insertion of $e$, the parity condition for $f + e$ guarantees a
  further valency outside $\pi$. Thus, the demand of $\pi$ is
  satisfied.

  \emph{Case B:} $\pi \subseteq S$.  Then $S$ is either a path of at
  least three or a cycle of at least four assigned degree-2 vertices in
  $f$ that are incident to two distinct faces. 
  Thus, $e$ is chosen by one of the
  first subrules of Rule~\ref{r:2}. 
  Recall that $S$ looses one valency due to the insertion. 
  We distinguish whether~$\pi$ is a
  path or a cycle.
  
  Suppose $S$ is a path of length $k \geq 3$ in Phase~1, where $f$ is
  incident to distinct connected components.  In this case the
  insertion of $e$ splits $S$ into two paths $\pi_1$ and $\pi_2$, both
  of length $(k-1)/2$ if $k$ is odd, and $|\pi_1 | = k/2$ and $|\pi_2|
  = k/2 -1$ if $k$ is even. It is $\pi = \pi_1$.
  In the latter case, if $k$ is even, the parity condition for $f$
  guarantees a further valency outside $S$ besides the valency at
  $v$. Thus, $\pi_2$ together with this valency satisfies the demand
  of $\pi$. Analogously, $\pi$ and $\pi_2$ mutually satisfy their
  demands if $k$ is odd.

  Now suppose $S$ is a path of length $k \geq 3$ in Phase~2, where $f$ is
  incident to one component.  Then~$u$ is chosen as an endvertex of
  $S$ and the resulting path $\pi$ demands $k-1$ valencies in $f + e$.
  Since the planarity condition is satisfied for $f$, the demand of
  $S$ in $f$ is satisfied by~$k$ valencies outside $S$. Recall that
  the demand of $S$ can not be satisfied by a valency from a different
  component, since there are no valencies assigned to distinct
  components in Phase~2. Thus, in $f+e$ remain at least $k-1$
  valencies outside $S$, and hence, outside $\pi$, which satisfy the
  demand of $\pi$.

  Suppose $S$ is a cycle of length $k \geq 4$ in Phase~1.  In this case $u\in
  S$ is connected to a vertex $v$ at a different component, and $S$
  becomes a path of length $k-1 \geq 3$, which is $\pi$.  Since the planarity
  condition is satisfied for $f$, the demand of $S$ in $f$ is
  satisfied by two further valencies from two further components or by
  $k$ assigned valencies outside $S$.  In the first case remains at
  least one component in $f+e$, which satisfies the demand of $\pi$.
  In the second case remain at least $k-1$ valencies outside $S$ in
  $f+e$, which satisfy the demand of $\pi$.  Phase~2 considers no
  cycles.

  {\bf Case II:} $e\subseteq \pi$.  In order to create a new path (or
  cycle) of assigned degree-2 vertices that are incident to two
  distinct faces, $e$ must connect two degree-1 vertices from the same
  component. Thus, the only rule possibly choosing such an edge is the
  second subrule in Phase~2. Note that in Phase~2 it is $\kmax \leq
  3$, since any indicator set of demand~$4$ would induce an additional
  connected component.
  If this rule is applied for $\kmax = 2$, the longest path (or cycle)
  that can occur consists of two vertices, which is no path (or cycle)
  as considered in the planarity condition. If this rule is applied
  for $\kmax = 3$, $S$ is a branch, and the rule connects the degree-1
  vertex $u \in S$ to a second degree-1 vertex $v$. This yields a path
  $\pi$ of length 3. Note that creating a cycle in this way is not
  possible, since $f$ is incident to only one component.

  Suppose the rule creates a new path $\pi$ of length 3.  Then no
  feasible degree-2 vertex could be reached by a left first or right
  first search from $u\in S$. Otherwise, the rule would have connected
  $u$ to this vertex. This is, the degree-2 vertex $r\in S$ is
  adjacent to a degree-3 vertex and the first valency found in the
  opposite direction from $v$ also belongs to a degree-1 vertex
  $w$. Thus, there are at least $3 + 2 + 2 = 7$ valencies from $S$,
  $v$ and $w$ assigned to $f$. Due to the parity condition for $f$
  there is a further assigned valency outside $S$. In $f+e$ this
  valency together with $w$ satisfies the demand of the newly created
  path $\pi$.

  Finally we prove that the \emph{matching condition} is satisfied in
  $f+e$. The second subrules of Rule~\ref{r:2} inherit these property
  from Rule~\ref{r:1}. Thus, we focus on the first subrules, where $S$
  is a path (or cycle) of length $k >\kmax$ in $f$.  In order to prove
  the matching condition let $S'$ denote a maximum indicator set in
  $f+e$.  We prove that the demand of $S'$ is satisfied in $f+e$. Then
  the matching condition is satisfied for $f+e$, according to
  Lemma~\ref{lem:max-indicator} and Theorem~\ref{thm:characMatch}.

  Note that the insertion of $e$ does not create any new indicator
  set, since $u$ becomes a degree-3 vertex in $f+e$.  Observe further
  that $S'\cap S = \emptyset$, unless $S'$ is a pair, which indicates
  $\kmax = 2$.

  First suppose $S'\cap S = \emptyset$ and recall that $k>\kmax$.
  This is, $S$ still provides $\kmax$ valencies at $\kmax$ degree-2
  vertices in $f+e$.  Obviously, this satisfies the demand of $S'$ in
  $f+e$.

  Now suppose $S' \subseteq S$ is a pair.  If $k \geq 5$, $S$ provides
  at least two valencies outside $S'$ in $f+e$, which satisfy the
  demand of $S'$.  If $k = 4$, $S$ provides one valency outside $S'$
  in $f+e$. However, the parity condition for $f+e$ guarantees a
  further valency outside $S$, such that the demand of $S'$ is
  satisfied.  The case $k = 3$ occurs only in Phase~2. In Phase~1
  $k=3$ would not yield a pair $S'\subseteq S$ in $f+e$.  Thus, if $k =
  3$, $S$ becomes $S'$ in $f+e$, and the planarity condition for $f$
  guarantees at least three further valencies outside $S$. Thus, in
  $f+e$ remain at least two valencies outside $S'$, which satisfy the
  demand of $S'$.
\end{proof}
% ------------------------- Ende Proof ----------------------------

Given a node assignment~$A$ and a face~$f$ satisfying the parity, 
matching, and planarity condition,
iteratively picking edges according to Rule~\ref{r:2} hence yields a
planar realization of~$A$ for~$f$.  Applying this to every face yields
the following theorem.

\newcommand{\theoremPlanarRealizable}{There exists a planar
  realization $W$ of $A$ if and only if $A$ satisfies for each face
  the parity, matching, and planarity condition; $W$ can be computed
  in $O(n)$ time.}

\begin{theorem}
  \label{thm:planar-realizable}
  \theoremPlanarRealizable
\end{theorem}
%
% ------------------------- Proof eingefuegt ----------------------------
\begin{proof}
  We construct the planar realization for each face individually.  To
  construct a local realization for a face with a positive number of
  assigned vertices, we repeatedly apply Rule~\ref{r:2} to select an
  edge.  By Lemma~\ref{lem:R2} this yields a planar local realization.
  It is not hard to see that repeatedly applying Rule~\ref{r:2} for a
  face~$f$ can be done in time proportional to the number of vertices
  incident to~$f$.  To allow fast left-first and right-first searches,
  we maintain a circular list containing the vertices incident to~$f$
  with degree less than~3, and remove vertices reaching degree~3 from
  this list.  Thus, also in phase~2 of Rule~\ref{r:2} the second
  vertex can be found in~$O(1)$ time.
\end{proof}
% ------------------------ Ende Proof ------------------------------------

\subsection{Globally Realizable Node Assignments and Planarity}

In this section we show how to compute a node assignment that is
realizable in a planar way if one exists.  By
Theorem~\ref{thm:planar-realizable}, this is equivalent to finding a
node assignment satisfying for each face the parity, matching, and
planarity condition.  In a first step, we show that the planarity
condition can be neglected as an assignment satisfying the other two
conditions can always be modified to additionally satisfy the
planarity condition.

\newcommand{\lemmaMatchingPlanarity}{Given a node assignment $A$ that
  satisfies the parity and matching condition for all faces, a node
  assignment $A'$ that additionally satisfies the planarity condition
  can be computed in $O(n)$ time.}

\begin{lemma}
  \label{lem:matching-planarity}
  \lemmaMatchingPlanarity
\end{lemma}
% ------------------------- Proof eingefuegt ------------------
\begin{proof}
  Assume that~$f$ is a face for which the planarity condition is not
  satisfied, and let~$\pi=v_1,\dots,v_k$ denote a largest path (or
  cycle) of degree-2 vertices, all assigned to~$f$, that violates the
  planarity condition.  Let~$f'$ denote the other face (distinct
  from~$f$) incident to~$\pi$.  Let~$u =v_1$.  Choose~$v=v_3$
  if~$k=3$, and $v = v_{\lceil (k+1)/2 \rceil}$ otherwise.  We
  modify~$A$ by reassigning $u$ and~$v$ to~$f'$.  We claim that this
  reassignment has two properties, namely 1) $f'$ satisfies exactly
  the same conditions as before the reassignment, and 2) $f$ satisfies
  the parity condition, the matching condition and the planarity
  condition.

  Note that since~$\pi$ is either a path of length more than~2 or a
  cycle of length more than~3, the two vertices~$u$ and~$v$ are
  distinct and non-adjacent.  To see property 1) consider the new
  assignment.  Obviously, the reassignment preserves the parity
  condition.  For the matching condition assume that~$M$ is an
  augmentation of~$f'$ with respect to~$A$.  Then~$M \cup \{uv\}$ is
  an augmentation of~$f'$ with respect to~$A'$, thus the matching
  condition is preserved.  Moreover, if~$M$ is a planar augmentation,
  then~$uv$ can be added in a planar way, showing that also the
  planarity condition is preserved.

  Concerning property 2), the reassignment obviously preserves the
  parity condition for~$f$.  For the matching and the planarity
  condition assume that there exists a set~$T$ of vertices assigned to
  $f$ that demand~$k'$ additional free valencies by either the
  matching condition or the planarity condition.  First observe
  that~$k' \le k-1$, as~$\pi$ would not have violated the planarity
  condition, otherwise.  If~$k' = k-1$, then $T$ is disjoint
  from~$\pi$, which provides~$k-2$ free valencies (recall that~$u$
  and~$v$ have been reassigned), and the parity condition implies the
  existence of an additional free valency assigned to~$f$, thus
  ensuring that the demand of~$T$ is satisfied.  The same argument
  works for all cases where~$T$ is disjoint from~$\pi$.  Thus assume
  that~$T$ and~$\pi$ are not disjoint.  Since~$\pi$ was chosen as a
  maximal path or cycle, and all sets with demands that contain
  degree-2 vertices it follows that~$T$ is a subset of~$\pi$.  Note
  that the reassignment splits $\pi$ into two disjoint subpaths~$\pi_1$
  and~$\pi_2$ consisting of~$\lceil(k-2)/2\rceil$ and~$\lfloor (k-2)/2
  \rfloor$ vertices, respectively.  Observe that~$\pi_2$, possibly
  together with an additional free valency provided by the parity
  condition (if~$k$ is odd) provides the necessary valencies
  for~$\pi_1$ and vice versa.  Thus the new assignment satisfies the
  matching condition and the planarity condition as well, and
  property~2) holds.

  Observe that once a largest path violating the planarity condition
  has been found, the reassignment for a face~$f$ takes only $O(1)$
  time.  Moreover, since we only need to consider maximal sequences of
  assigned degree-2 vertices, such a path can be found in time
  proportional to the size of~$f$.  The test whether the planarity
  condition for this path is satisfied can be performed in the same
  running time.  Thus~$A'$ can be computed from~$A$ by simply
  traversing all faces, spending time proportional to the face size in
  each face.  Thus, computing $A'$ from $A$ takes $O(n)$ time.
\end{proof}
% ------------------------Ende Proof -------------------------
\noindent
Lemma~\ref{lem:matching-planarity} and
Theorem~\ref{thm:planar-realizable} together imply the following
characterization.
\begin{theorem}
  $G$ admits a planar 3-regular augmentation if and only if it admits
  a node assignment that satisfies for all faces the parity and
  matching condition.
\end{theorem}

To find a node assignment satisfying the parity and matching
condition, we compute a (generalized) perfect matching in the
following (multi-)graph~$G_A=(\sV,E')$, called \emph{assignment
  graph}.
It is defined on~$\sV$, and the demand of a vertex in~$\ii V$ is~$3-i$
for~$i=0,1,2$.  For a face~$f$ let~$\sV_f \subseteq \sV$ denote the
vertices incident to~$f$.  For each face~$f$ of~$G$, $G_A$ contains
the edge set~$E_f = \binom{\sV_f}{2} \setminus E$, connecting
non-adjacent vertices in~$\sV$ that share the face~$f$.  We seek a
perfect (generalized) matching~$M$ of~$G_A$ satisfying exactly the
demands of all vertices.  The interpretation is that we assign a
vertex~$v$ to a face~$f$ if and only if~$M$ contains an edge incident
to~$v$ that belongs to~$E_f$.  It is not hard to see that for each
face~$f$ the edges in~$M \cap E_f$ are a (non-planar) realization of
this assignment, implying the parity condition and the matching
condition; the converse holds too.

\newcommand{\lemmaAssignmentGraph}{A perfect matching of~$G_A$
  corresponds to a node assignment that satisfies the parity and
  matching condition for all faces, and vice versa.}

\begin{lemma}
  \label{lem:assignment-graph}
  \lemmaAssignmentGraph
\end{lemma}
% -------------------- Proof eingefuegt --------------------
\begin{proof}
  First assume that~$M$ is a perfect matching of~$G_A$, and let~$A$ be
  the corresponding assignment.  Observe that for each face~$f$, the
  edge set~$E_f \cap M$ is exactly a realization of~$A$ for~$f$, and
  hence, by Theorem~\ref{thm:characMatch},~$A$ satisfies the parity
  condition and the matching condition for~$f$.  Conversely, again by
  Theorem~\ref{thm:characMatch}, for a node assignment~$A$ that
  satisfies the parity condition and the matching condition for each
  face~$f$, we find a realization $W_f$ for each face.  Note that by
  definition of~$E_f$ we have~$W_f \subseteq E_f$, and
  thus~$\bigcup_{f \in \F} W_f$ yields a perfect matching of~$G_A$
  inducing~$A$.
\end{proof}
% --------------------- Ende Proof -----------------------
%
Since testing whether the assignment graph admits a perfect matching
can be done in $O(n^{2.5})$ time~\cite{g-ertdc-83}, this immediately
implies the following theorem.
\newcommand{\theoremFERA}{\faug can be solved in~$O(n^{2.5})$ time.}
\begin{theorem}
  \label{thm:fera}
  \theoremFERA
\end{theorem}
% ------------------ Proof eingefuegt ---------------------
\begin{proof}
  For a given planar input graph~$G$ with $n$ vertices we first
  construct the assignment graph~$G_A=(\sV,E')$ in~$O(n^2)$ time.  We
  then check whether~$G_A$ admits a perfect matching in
  $O(\sqrt{|\sV|}|E'|)= O(n^{2.5})$ time, using an algorithm due to
  Gabow~\cite{g-ertdc-83}.  If no perfect matching exists, then~$G$
  does not admit a planar 3-regular augmentation by
  Lemma~\ref{lem:assignment-graph}.  Otherwise, we obtain by the same
  lemma a node assignment~$A$ that satisfies the parity condition and
  the matching condition for each face.  Using
  Lemma~\ref{lem:matching-planarity} we obtain in~$O(n)$ time a node
  assignment~$A'$ that additionally satisfies the planarity condition
  for each face.  A corresponding planar realization of~$A'$ can then
  be obtained in $O(n)$ time by Theorem~\ref{thm:planar-realizable}.
\end{proof}
% ----------------- ende Proof ----------------------------

%\vspace{-3ex}
\section{$C$-connected \faug}
\label{sec:planar-fixed-conn}
%\vspace{-1ex}
In this section we generalize the results obtained for $\faug$ to efficiently solve $c$-connected $\faug$ for $c = 1,2$. 
The triconnected case is shown to be NP-hard in Section~\ref{sec:tricon}.
We start with the connected case.

\subsection{Connected $\faug$}
\label{sec:connected}
Observe that an augmentation makes~$G$ connected if and only if in
each face all incident connected components are connected
by the augmentation.  We characterize the node
assignments admitting such \emph{connected realizations} and modify the assignment graph from the previous section to yield such
assignments.

Let $G=(V,E)$ be a planar graph with a fixed planar embedding, let~$f$ be a face of~$G$, and let~$z_f$ denote the number of connected components
incident to~$f$.  Obviously, an augmentation connecting all these
components must contain at least a spanning tree on these components,
which consists of~$z_f-1$ edges.  Thus the following \emph{connectivity condition} is 
necessary for a node assignment to admit a connected
realization for~$f$.

%\vspace{-1ex}
\begin{condition}[connectivity] 
\begin{compactenum}[(1)]
\itemsep=0ex
\item If~$z_f > 1$, each connected component incident to~$f$ must have
  at least one vertex assigned to~$f$.
\item The number of valencies assigned to~$f$ must be at least~$2z_f-2$.
 \end{compactenum}
\end{condition}
%\vspace{-1ex}
%
  It is not difficult to see that this condition is also sufficient
  (both in the planar and in the non-planar case) since both
  Rule~\ref{r:1} and Rule~\ref{r:2} gives us freedom to choose the
  second vertex~$v$ arbitrarily. We employ this degree of freedom
  to find a connected augmentation by choosing~$v$ in a connected
  component distinct from the one of~$u$, which is always possible due
  to the connectivity condition.

  \newcommand{\theoremConnAug}{There exists a connected realization
    $W$ of $A$ if and only if~$A$ satisfies the parity, matching, and
    connectivity condition for all faces. Moreover, $W$ can be chosen
    in a planar way if and only if~$A$ additionally satisfies the
    planarity condition for all faces.  Corresponding realizations can
    be computed in~$O(n)$ time.}

\begin{theorem}
  \label{thm:conn-aug}
  \theoremConnAug
\end{theorem}
%
% --------------------------- Proof eingefuegt -------------------
\begin{proof}
  Clearly the conditions for both statements are necessary.  We prove
  that they are also sufficient.  Let~$A$ be a node assignment
  satisfying the parity condition, the matching condition, and the
  connectivity condition for all faces of~$G$.  We construct a
  connected realization of~$A$ for each face~$f$; together they form a
  connected realization of~$A$.

  To construct a connected (possibly non-planar) realization for~$f$,
  we repeatedly choose edges according to Rule~\ref{r:1} (which yields
  a realization by Lemma~\ref{lem:R1}), making use of the freedom in
  the rule to reduce the number of connected components.
  Rule~\ref{r:1} prescribes one endpoint~$u$ of the edge that will be
  selected, and we are free to choose~$v\in\sV$ arbitrarily, as long
  as it is not incident to~$u$. We then choose~$v$ in a connected
  component different from the one containing~$u$ as long as several
  connected components exist.  While the number~$z_f$ of connected
  components incident to~$f$ is greater than~2, there exists a
  connected component assigning at least two valencies to~$f$, due to
  connectivity condition~(2).  We choose~$v$ such that at least one
  of~$u$ and~$v$ is contained in such a connected component.  We then
  consider the node assignment~$A$ for~$f+e$.  By
  Lemma~\ref{lem:R1}~$A$ satisfies the parity condition and the
  matching condition for~$f+e$.  Moreover, our choice of~$v$ ensures
  that after adding the edge~$uv$ determined by the rule, connectivity
  condition~(1) is satisfied for the resulting connected component.
  Connectivity condition~(2) is trivially preserved, showing that~$A$
  satisfies the connectivity condition for~$f+e$.  Thus, the
  construction can be repeated, eventually yielding a connected
  realization for~$f$.  The planar case works completely analogously,
  using Rule~\ref{r:2} and Lemma~\ref{lem:R2} instead of
  Rule~\ref{r:1} and Lemma~\ref{lem:R1}.  The running time can be
  argued as in the proof of Theorem~\ref{thm:planar-realizable}.
\end{proof}
% --------------------------- ende Proof ---------------------------
%
The following corollary follows from Theorem~\ref{thm:planar-realizable} by showing that the reassignment which establishes the
planarity condition preserves the connectivity condition.
\newcommand{\corConnectedPlanarRealizable}{Given a node assignment $A$
  that satisfies the parity, matching and connectivity condition for
  all faces, a node assignment $A'$ that additionally satisfies the
  planarity condition can be computed in $O(n)$ time.}
\begin{corollary}
  \label{cor:connected-planar-realizable}
  \corConnectedPlanarRealizable
\end{corollary}
% ------------------------------Proof eingefuegt -------------------
\begin{proof}
  To see this, recall that Theorem~\ref{thm:planar-realizable}
  reassigns from each face at most two vertices to a distinct face if
  the planarity condition is not satisfied.  Clearly, assigning more
  vertices to a face does not invalidate the connectivity condition.
  Thus, an invalidation of the connectivity condition for a face~$f$
  may only happen when two vertices assigned to~$f$ are reassigned to
  a different face.  Note that if~$z_f > 2$, the planarity condition
  is implied by connectivity condition (1).  Thus a reassignment only
  happens for faces with~$z_f=1,2$.  If~$z_f=1$, the connectivity
  condition holds trivially.  If~$z_f=2$, observe that connectivity
  condition (2) is implied by condition (1), and since the
  reassignment does not reassign the last valency of a connected
  component, connectivity condition (1) is preserved
\end{proof}
%------------------------------ Ende Proof -----------------------
Corollary~\ref{cor:connected-planar-realizable} and
Theorem~\ref{thm:conn-aug} together imply the following
characterization.
\begin{theorem}
  $G$ admits a connected planar 3-regular augmentation iff it admits a
  node assignment that satisfies the parity, matching and connectivity
  condition for all faces.
\end{theorem}
We describe a modified assignment graph, the \emph{connectivity
  assignment graph}~$G_A'$, whose construction is such that there is a
correspondence between the perfect matchings of~$G_A'$ and node
assignments satisfying the parity, matching and connectivity
condition.

To construct the connectivity assignment graph a more detailed look at
the faces and how vertices are assigned, is necessary.  A
\emph{triangle} is a cycle of three degree-2 vertices in~$G$.  An
\emph{empty triangle} is a triangle that is incident to a face that
does not contain any further vertices.  The set~$\Vin$ (for inside)
contains all vertices from $\0{V}\cup\1{V}$, all degree-2 vertices
incident to bridges (they are all incident to only a single face), and
all vertices of empty triangles (although technically they are
incident to two faces, no augmentation edges can be embedded on the
empty side of the triangle).  We call the set of remaining
vertices~$\Vbound$ (for boundary).  We construct a preliminary
assignment~$\preA$ that assigns the vertices in the set~$\Vin$ of~$G$
whose assignment is basically unique.  The remaining degree of freedom
is to assign vertices in~$\Vbound$ to one of their incident faces.
The connectivity assignment graph~$G_A'$ again has an edge set~$E_f'$
for each face~$f$ of~$G$.  Again the interpretation will be that a
perfect matching~$M$ of~$G$ induces a node assignment by assigning
to~$f$ all vertices that are incident to edges in~$M \cap E_f'$.

\begin{figure}
  \centering
  \includegraphics{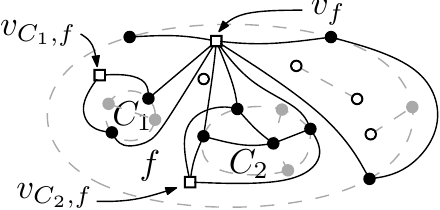}
  \caption{Graph (dashed lines; preassigned vertices are empty) and
    its connectivity assignment graph (solid lines, dummy vertices as
    boxes).}
  \label{fig:con_graph}
\end{figure}
If a face~$f$ is incident to a single connected component, we use
for~$E_f'$ the ordinary assignment graph; the connectivity condition
is trivial in this case.  Now let~$f$ be a face with~$z_f>1$ incident
connected components.  For each component~$C$ incident to~$f$ that
does not contain a vertex that is preassigned to~$f$, we add a dummy
vertex~$v_{C,f}$ with demand~1 and connect it to all degree-2 vertices
of~$C$ incident to~$f$; this ensures connectivity condition~(1).
Let~$c_f$ denote the number of these dummy vertices, and note that
there are exactly~$c_f$ valencies assigned to~$f$ due to these dummy
vertices.  Let~$\prea$ denote the number of free valencies assigned
by~$\preA$, and let~$\predem$ denote the number of valencies a maximum
indicator set in~$f$ with respect to~$\preA$ misses.  To ensure that
the necessary valencies for the matching condition are present, we
need that at least~$\predem - c_f$ vertices of~$\Vbound$ are assigned
to~$f$.  For connectivity condition~(2) we need at least~$2z_f - 2
-\prea - c_f$ such vertices assigned to~$f$.  We thus create a dummy
vertex $v_f$ whose demand is set to~$s_f = \max\{2z_f - 2 -\prea -
c_f, \predem - c_f, 0\}$, possibly increasing this demand by~1 to
guarantee the parity condition.  Finally, we wish to allow an
arbitrary even number of vertices in~$\Vbound$ to be assigned to~$f$.
Since some valencies are already taken by dummy vertices, we do not
just add to~$E_f'$ edges between non-adjacent vertices of~$\Vbound$
incident to~$f$ but for all such pairs.  The valencies assigned
by~$\preA$ and the dummy vertices satisfy the demand of any indicator
set.  Fig.~\ref{fig:con_graph} shows an example; for clarity edges
connecting vertices in~$V_b$ are omitted in~$f$ and the outer face. 

\newcommand{\lemmaConnectivityAssignmentGraph}{A perfect matching
  of~$G_A'$ (together with $\preA$) corresponds to a node assignment
  that satisfies parity, matching, and connectivity condition for all
  faces, and vice versa.}

\begin{lemma}
 \label{lem:connectivity-assignment-graph}
  \lemmaConnectivityAssignmentGraph
\end{lemma}
% ------------------------ Proof eingefuegt -------------------
\begin{proof}
  Let~$M$ be a perfect matching of~$G_A'$ and let~$A$ denote the
  corresponding node assignment.  Let~$f$ be a face of~$G$, we show
  that~$A$ satisfies the parity condition, the matching condition, and
  the parity condition for~$f$.  If~$z_f=1$, the connectivity
  condition holds trivially and the remaining conditions follow from
  Theorem~\ref{thm:characMatch} since~$M \cap E_f$ is a realization
  of~$A$ for~$f$.  Hence, let~$z_f \ge 2$.

  Using the definition from above, there are~$\prea$ valencies
  assigned to~$f$ by~$\preA$,~$c_f$ valencies from vertices adjacent
  to the dummy vertices~$v_{C,f}$,~$s_f$ valencies from vertices
  incident to the dummy vertex~$v_f$ and $2k_f$ valencies from~$k_f$
  edges in $M \cap \binom{X_f}{2}$.  In total this are~$\prea + c_f +
  s_f + 2k_f$ valencies, which is even due to the choice of~$s_f$, and
  hence the parity condition holds.

  For the connectivity condition, observe that the dummy
  vertices~$v_{C,f}$ imply connectivity condition~(1) and the choice
  of~$s_f$ implies connectivity condition~(2).

  It remains to prove that the matching condition is satisfied.
  Let~$T$ denote an indicator set of~$f$ (for~$A$).  Observe that the
  vertices of an indicator set are either all in~$\Vin$ or all in
  $\Vbound$.  If~$T \subseteq \Vin$, then~$T$ was already an indicator
  set for~$\preA$, and its demand is satisfied due to the choice
  of~$s_f$.  If~$T \subseteq \Vbound$, it is a joker, a pair, or a
  3-cycle.  However, as argued before, a 3-cycle can be excluded as it
  is either contained in~$\Vin$ or one of its vertices must be matched
  to a dummy vertex in another face, and hence is not assigned to~$f$.
  For a joker the necessary valency exists due to the parity
  condition.  If~$T$ is a pair (consisting of two adjacent vertices of
  degree~2), its vertices are contained in the same connected
  component.  Since~$z_f \ge 2$ and connectivity condition (1) is
  satisfied, at least one more vertex must be assigned to~$f$.  It
  then follows from the parity condition, that the demand of~$T$ is
  satisfied.  Thus the matching condition holds, finishing this
  direction of the proof.

  Conversely, let~$A$ be a node assignment that satisfies for each
  face~$f$ the parity condition, the matching condition, and the
  connectivity condition.  We construct for each face~$f$ a
  matching~$M_f \subseteq E_f$ satisfying exactly the demands of all
  vertices assigned to~$f$ and the dummy vertices associated with~$f$.
  Clearly the matching~$M = \bigcup_{f \in \mathcal{F}} M_f$,
  where~$\mathcal{F}$ denote the set of faces of~$G$, then satisfies
  the demands of all vertices in~$G_A'$, that is it is a perfect
  matching of~$G_A'$.

  Let~$f$ be a face.  If~$z_f = 1$, we choose~$M_f$ as an arbitrary
  realization of~$A$ for~$f$, which exists by
  Theorem~\ref{thm:characMatch}, and the condition is satisfied by
  construction.  Hence assume~$z_f \ge 2$.  Connectivity condition~(1)
  implies that each connected component either contains a vertex
  in~$\Vin$ or a vertex in~$\Vbound$ assigned to~$f$.  We pick for
  each connected component~$C$ that does not contain a vertex
  in~$\Vin$ an arbitrary assigned vertex of~$\Vbound$ and match it
  to~$v_{C,f}$.  The matching condition implies that the number~$r$ of
  remaining vertices in~$\Vbound$ assigned to~$f$ is at least~$\predem
  - c_f$, and connectivity condition~(2) implies that~$r$ is at
  least~$2z_f-2-\prea-c_f$.  Thus, we can match arbitrary~$s_f$
  vertices in~$\Vbound$ to~$v_f$, satisfying its demand.  The
  remaining yet-unmatched vertices assigned to~$f$ are an even number
  and an arbitrary pairing of them completes the matching~$M_f$.
\end{proof}
% ------------------------ ende Proof --------------------------

Together with the previous observations this directly implies
an algorithm for finding connected 3-regular augmentations.

\begin{theorem}
  Connected \faug can be solved in~$O(n^{2.5})$ time.
\end{theorem}

\subsection{Biconnected \faug}
\label{sec:biconnected-appendix}

In this section we show that also biconnected \faug can be solved
efficiently.  Again, we first give a local characterization of node
assignments admitting biconnected augmentations and then construct a
\emph{biconnectivity assignment graph} whose perfect matchings
correspond to such node assignments.

\paragraph{Local characterization of biconnectivity.}

Let~$G=(V,E)$ be a planar graph with a fixed embedding and let~$f$ be
a face of~$G$.  We consider the \emph{bridge forest}~$B_f$ of~$f$,
which is constructed as follows.  Remove all bridges from~$G$ and
consider the connected components of this graph that are incident
to~$f$.  We create a node for each such connected component and
connect them by an edge if and only if they are connected by a bridge
in~$G$.  Similarly, we can define the bridge forest of~$f$ with
respect to an augmentation~$W$, where we only remove bridges of~$G+W$.
Observe that each leaf component in a bridge forest with respect to $W$ contains a 
subgraph that corresponds to a leaf component in the associated bridge forest of $G$.
Clearly, an augmentation is connected if and only if the bridge graph
of each face is connected, and it is biconnected if and only if each
bridge forest consists of a single node.  Observe that the bridge
forest~$B_f$ contains a connected component for each connected
component of~$G$ incident to~$f$.  We say that such a component is
\emph{trivial} if its corresponding connected component in~$B_f$
consists of a single node.  A 2-edge connected component of~$G$
incident to~$f$ is a \emph{leaf component} if its corresponding node
in~$B_f$ has degree~1. Figure \ref{fig:bridge_forest} shows an example.

\begin{figure}[tb]
    \centering
    \includegraphics{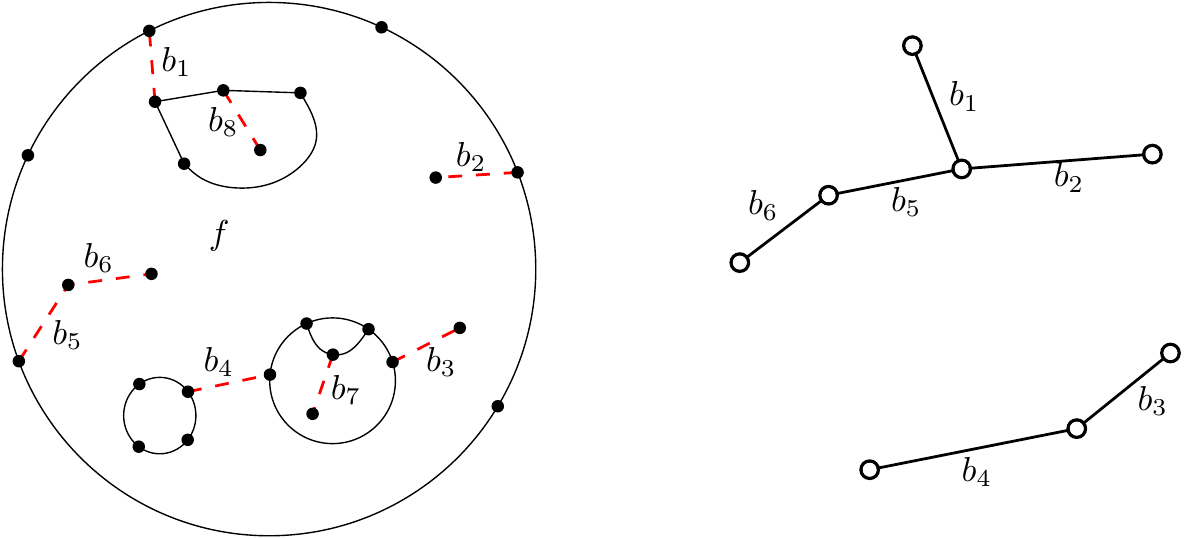}
    \caption{A face~$f$ (right) and its corresponding bridge
      forest~$B_f$ (right); bridges are dashed, red edges.  The
      bridges $b_7$ and $b_8$ are not incident to $f$ and hence not
      contained in~$B_f$.}
    \label{fig:bridge_forest}
  \end{figure}

Next, we study necessary and sufficient conditions for when a node
assignment~$A$ admits for a face~$f$ a planar 3-regular
augmentation~$W_f$ such that the resulting bridge forest is a single
node.  Obviously, if there is more than one connected component
incident to~$f$, each of them must assign at least two valencies
to~$f$; if none is assigned, the augmentation will not be connected,
if only one is assigned the single edge incident to this valency will
form a bridge.  Additionally, each leaf component must assign at least
one valency, otherwise its incident bridge in~$B_f$ will remain a
bridge after the augmentation.  Thus the following
\emph{biconnectivity condition} is necessary for a face~$f$ with~$z_f$
incident connected components to admit a biconnected augmentation.

\begin{condition}[Biconnectivity condition]\hfil
\begin{enumerate}[(1)]
  \itemsep=0ex
\item If~$z_f>1$, each connected component incident to~$f$ must have
  at least two valencies assigned to~$f$, and
\item each leaf component of~$f$ must assign at least one valency
  to~$f$.
\end{enumerate}
\end{condition}

We show that these conditions are also sufficient, both in the planar
and in the non-planar case.

\begin{theorem}
  \label{thm:biconn-aug}
  Let~$G$ be a planar maxdeg-3 graph on $n$ vertices with a fixed
  embedding, and let~$A$ be a node assignment.  Then the following
  statements hold.
  \begin{enumerate}[(i)]
  \itemsep = 0ex
\item $A$ admits a biconnected realization if and only if~$A$
  satisfies the parity condition, the matching condition, and the
  biconnectivity condition for all faces of~$G$.
\item The realization can be chosen to be planar if and only if~$A$
  additionally satisfies the planarity condition for faces of~$G$.
  A corresponding realization for~$A$ can be
  computed in~$O(n)$ time.
  \end{enumerate}
\end{theorem}

\begin{proof}
  Clearly the conditions for both statements are necessary.  We prove
  that they are also sufficient.  Let~$A$ be a node assignment
  satisfying the parity condition, the matching condition, and the
  biconnectivity condition for all faces of~$G$.  We construct a
  biconnected realization of~$A$ for each face~$f$; together they form
  a biconnected realization of~$A$.  

  First assume that the bridge graph~$B_f$ of~$f$ is connected, that
  is~$G$ has only one connected component incident to~$f$.  Let~$W_f$
  be a realization of~$A$ for~$f$ and assume that it is not a
  biconnected realization.  Then there exists a bridge~$b$ in~$G+W_f$
  whose endpoints are incident to~$f$ in~$G$.  Without loss of
  generality we assume~$b$ such that it is incident to a leaf of the
  bridge forest $B_f$ with respect to $W_f$.  Consider the subtrees of~$B_f$ 
  on distinct sides
  of~$b$.  Each of them contains a leaf, and thus a vertex assigned
  to~$f$.  Let~$x_1,x_2$ denote two such vertices on distinct sides
  of~$b$, choosing~$x_1$ and~$x_2$ as endpoints of~$b$ if possible,
  and let~$e_1=x_1y_1$ and~$e_2=x_2y_2$ denote two edges of~$W_f$
  incident to~$x_1$ and~$x_2$, respectively.  Since~$b$ is a bridge
  in~$G+W_f$, we have that~$x_1,x_2,y_1$ and~$y_2$ are pairwise
  disjoint, and except for possibly~$x_1$ and~$x_2$, which might be
  joined by~$b$, they are pairwise non-adjacent.  It is then not hard
  to see that replacing in~$W_f$ the edges~$x_1y_1$ and~$x_2y_2$
  by~$x_1y_2$ and~$x_2y_1$ yields a new augmentation~$W_f'$ of~$G$
  whose bridge forest has fewer leaves.  Applying this construction
  iteratively yields a biconnected realization.  Now assume that the
  bridge graph of~$f$ is not connected and consists of~$z_f$ connected
  components.  By correctness of Rule~\ref{r:1}, we may first add a
  set~$W_f''$ of~$z_f-1$ edges such that~$G$ becomes connected.
  Observe that~$G+W_f''$ is still a planar graph and contains the
  face~$f$.  Consider the assignment~$A'$ induced by~$A$ on~$G+W_f''$.
  The bridge forest~$B_f'$ of~$G+W_f''$ for~$f$ is connected, and each
  leaf component of $G+W_f''$ either is a connected component of $G$,
  or was already a leaf component in~$G$.  In both cases it follows
  that the leaf component contains at least one vertex assigned
  to~$f$, which is incident to an edge from~$W_f''$ that is not a
  bridge.  Hence the construction also works for the case that~$B_f$
  is not yet connected, finishing the proof of claim~(i).
  
  For claim~(ii) observe that by the same argument (applying Rule~\ref{r:2})
  instead of Rule~\ref{r:1}), we may again assume that~$B_f$ is connected.
  Now let~$W_f$ be any planar realization of $A$ for a face $f$ of~$G$ and let~$b$
  be a bridge that is incident to a leaf of~$B_f$ (with respect to $W_f$).  We again choose
  vertices~$x_1,x_2,y_1,y_2$ for rewiring, but slightly more
  carefully.  Namely observe that in~$G+W_f$ the bridge~$b$ is
  incident to a face~$f'$ in~$G+W_f$, and by the same argument as
  above this face must be bounded by at least two edges~$x_1y_1$
  and~$x_2y_2$ of~$W_f$, having their endpoints on distinct sides
  of~$b$.  We assume that the clockwise order of occurrence along the
  boundary of~$f'$ is~$x_1x_2y_2y_1$.  If~$x_1$ and~$x_2$ are not
  adjacent, we replace~$x_1y_1$ and~$x_2y_2$ by the two edges~$x_1x_2$
  and~$y_1y_2$, which clearly is planar.  If~$x_1$ and~$x_2$ are
  adjacent, that is~$b=x_1x_2$, we replace them by~$y_2x_1$
  and~$x_2y_1$, which is again planar since~$b$ is incident to~$f'$ on
  both sides; see Fig.~\ref{fig:bicon_rewire}.  As above it can be seen that the
  bridge graph of~$G+W_f$ has fewer leaves, and thus iteratively
  applying the rewiring step yields the desired realization.

\begin{figure}[tb]
  \centering
   \includegraphics{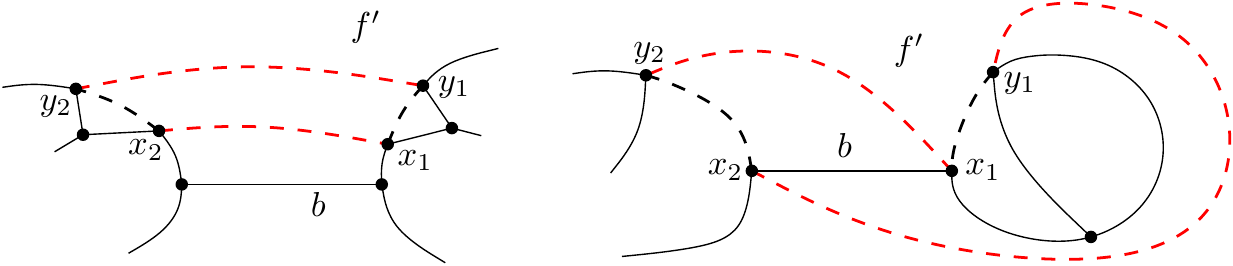}
   \caption{Illustration of the rewiring step from the proof of
     Theorem~\ref{thm:biconn-aug}; $x_1,x_2$ are not adjacent (left),
     and~$x_1,x_2$ are adjacent (right).  The dashed black edges are
     part of an augmentation and are replaced by the dashed red edges.
     Afterwards the edge~$b$ is not a bridge anymore.}
    \label{fig:bicon_rewire}
  \end{figure}

Concerning the running time, recall that, by applying Rule~\ref{r:2}, 
$z_f-1$ edges can be added to each face in $G$ in~$O(n)$ time 
such that $G$ becomes connected. A planar (possibly not biconnected) 
realization $W_f$ of $A$ for each 
face $f$ in the connected graph $G$ can be computed in linear time
by Theorem~\ref{thm:conn-aug}.
For the rewiring we apply a right-first search along 
the boundary of $f$ handling the bridges that are incident 
to leaves of the bridge forest consecutively.

We start from a vertex outside a leaf component and search for the 
first bridge $b$ that is incident to a leaf component 
of $B_f$ (with respect to $W_f$). 
After crossing $b$ in the direction of the leaf, 
we follow the boundary of the face $f'$ in $G + W_f$ that is incident to $b$ on both sides
(instead of the boundary of $f$ in $G$).
We store an edge $y_1x_1 \in W_f$ at the leaf component (which exists by the arguments above) and
the first edge $x_2y_2 \in W_f$ on the other side of $b$, i.e., after crossing the bridge the second time.
These edges are then rewired. If $x_2y_2$ was an edge in a leaf component before, there are no further
bridges connected to the boundary of $f$ between $x_2$ to $y_2$ and we continue the search 
for the next bridge that is incident to a leaf from $y_2$, now again along the boundary of $f$.
Otherwise, we recursively search
for the next bridge starting from $x_2$ along the boundary of $f$,
and apply the rewiring.
When we reach $y_2$ again, the previously rewired edge incident to $y_2$ might be 
rewired again, however, at this point all bridges between $x_2$ and $y_2$ are remedied.
The search continues at~$y_2$ and finally stops at~$y_1$ at the latest. 
Then~$B_f$ consists of a single node.
Hence, this search and the rewiring can be done in linear time with respect to 
the number of vertices incident to $f$.
\end{proof}

Next, we extend Theorem~\ref{thm:planar-realizable} to this setting,
allowing us to additionally enforce the planarity condition for an
assignment already satisfying the parity condition, the matching
condition and the biconnectivity condition.  Similar to the proof of
Corollary~\ref{cor:connected-planar-realizable}, it can be seen that
the rewiring performed in the proof of
Theorem~\ref{thm:planar-realizable} does not invalidate the
biconnectivity condition.  It reassigns vertices to other faces only
if a face is assigned an insufficient number of additional valencies,
which only shortens long paths of degree-2 vertices but never reduces
the number of assigned valencies of a leaf component to zero or of a connected
component below two.  We thus have the following corollary.

\begin{corollary}
  \label{cor:biconnected-planar-realizable}
  Let~$G$ be a planar maxdeg-3 graph with a fixed embedding and
  let~$A$ be a node assignment satisfying the parity condition, the
  matching condition, and the biconnectivity condition.  Then a
  modified assignment~$A'$ that additionally satisfies the planarity
  condition can be computed from~$A$ in~$O(n)$ time.
\end{corollary}

\paragraph{The biconnectivity assignment graph.}

We now show how to test efficiently whether such an augmentation
exists.  The idea is, as in the connected case, to consider a
corresponding \emph{biconnectivity assignment graph} that models the
additional requirements.  Let~$G = (V,E)$ be a planar maxdeg-3 graph
with a fixed embedding.  We construct the biconnectivity assignment
graph~$G_A''$ of~$G$ using similar techniques as for the connectivity
assignment graph.  In particular, we again consider the partition of
the vertices of $\sV$ into the sets~$\Vin$ with a fixed assignment
and~$\Vbound$ that may be assigned to two different faces, and the
corresponding preliminary assignment~$\preA$.  We note that unlike in
the connectivity case, a graph that contains a non-empty triangle
generally does not admit a biconnected augmentation as one of the
edges incident to such a triangle necessarily forms a bridge.  Hence
if~$G$ admits a biconnected 3-regular augmentation, all triangles are
empty, and thus assigned by~$\preA$.

Again the biconnectivity assignment graph~$G_A'' = (V'',E'')$ is
formed on a superset~$V'' \supseteq \Vbound$ of the vertices without a
fixed assignment.  We now describe an edge set~$E_f$ for each
face~$f$.  As before, the final interpretation will be that in the
assignment induced by a (generalized) perfect matching~$M$ of~$G_A''$
a vertex~$v \in\Vbound$ is assigned to~$f$ if and only if it is
incident to an edge in~$M \cap E_f$.  If a face~$f$ is incident to a
single connected component whose bridge forest consists of a single
node, we use the ordinary assignment graph, where~$E_f$ consists of
edges between all non-adjacent pairs of vertices from~$\sV$ incident
to~$f$.  Now assume that~$f$ does not have this property.  To enforce
the biconnectivity condition, we consider the leaf components of~$f$.
For each leaf component~$L$ that does not contain vertices preassigned
to~$f$, we add a dummy vertex~$v_{L,f}$ with demand~1 and connect it
to all vertices from~$\Vbound$ of~$L$ incident to~$f$; this clearly
enforces biconnectivity condition~(2).  Moreover, for each connected
component of~$f$ that contains a bridge this also enforces
biconnectivity condition~(1).  If~$G$ is not connected, we add for
each connected component~$C$ that neither contains a bridge nor any
vertices preassigned to~$f$, a dummy vertex~$v_{C,f}$ with demand~2
and connect it to all vertices of~$C$ that are in~$\Vbound$ and
incident to~$f$. Note that these components are no leafes by
definition.  Let~$c_f$ denote the number of valencies demanded by
dummy vertices~$v_{C,f}$ and~$\ell_f$ the number of valencies demanded
by dummy vertices~$v_{L,f}$.  We compute the demand~$\predem$ of the
preassigned vertices as in the construction of the connectivity
assignment graph.  To satisfy this demand, at least~$\predem - c_f -
\ell_f$ additional vertices from~$\Vbound$ need to be assigned to~$f$.
We thus set~$n_f = \max\{\predem - c_f - \ell_f,0\}$.  To ensure this
and the parity condition, we add a new dummy vertex~$v_f$ with
demand~$s_f$, where~$s_f = n_f$ if~$n_f + \prea + c_f + \ell_f$ is
even and~$s_f = n_f + 1$, otherwise.  Finally, we allow an arbitrary
even number of vertices of~$\Vbound$ incident to~$f$ to be matched by
adding to~$E_f$ all edges in~$\binom{X_f}{2}$, where~$X_f$ denotes the
vertices in~$\Vbound$ incident to~$f$.  The connectivity assignment
graph in Figure \ref{fig:con_graph} coincides with the biconnectivity
assignment graph for the given graph, except for the demands of
$v_{C_1,f}$ and $v_{C_2,f}$.  They are set to 2 here, because the
corresponding components neither contain a bridge nor any preassigned
vertices.  Recall that, for reasons of clarity, the edges between
vertices of~$\Vbound$ incident to the face $f$ are omitted.

\begin{lemma}
  \label{lem:biconnectity-assignment-graph}
  Let~$G$ be a maxdeg-3 graph with a fixed embedding and let~$G_A''$
  be its biconnectivity assignment graph.  Then each perfect
  matching~$M$ of~$G_A''$ induces (together with the preliminary
  assignment~$\preA$) a node assignment that satisfies for each face
  the parity condition, the matching condition, and the biconnectivity
  condition.  Conversely, for each such node assignment~$A$ there
  exists a perfect matching~$M$ of~$G_A'$ that induces~it.
\end{lemma}

\begin{proof}
  Let~$M$ be a perfect matching and let~$A$ be the corresponding node
  assignment.  For faces~$f$ that are incident to a single connected
  component whose bridge graph consists of a single node, $E_f$ is
  chosen as in the ordinary assignment graph, and thus~$A$ satisfies
  the parity condition and the matching condition for~$f$.  The
  biconnectivity condition holds trivially for these faces.  Now
  assume that~$f$ is a face not having this property.  Clearly, the
  demand of~$s_f$ ensures that an even number of valencies is assigned
  to~$f$, and hence~$A$ satisfies the parity condition.  Moreover, the
  dummy vertices~$v_{C,f}$ and~$v_{L,f}$ explicity ensure the
  biconnectivity condition.  For the matching condition, recall
  that~$s_f$ was chosen such that the demands of all matching
  indicator sets consisting of vertices in~$\Vin$ are satisfied.
  Thus, if there is a matching indicator set whose demand is not
  satisfied, it would consist of vertices in~$\Vbound$, and hence
  would be a joker, a pair or a 3-cycle.  Let~$T$ be such an indicator
  set.  That the demands of jokers and pairs are satisfied can be
  argued as in Lemma~\ref{lem:connectivity-assignment-graph}.  If~$T$
  is a 3-cycle, it forms a non-empty triangle (as it would be
  preassigned otherwise), but then, as argued above, a biconnected
  augmentation does not exists.  Hence this case cannot occur, and the
  matching condition is satisfied.

  Conversely let~$A$ be a node assignment satisfying the parity
  condition, the matching condition, and the biconnectivity condition.
  We construct for each face~$f$ a matching~$M_f \subseteq E_f$
  satisfying exactly the demands of all vertices assigned to~$f$ and
  the dummy vertices associated with~$f$.  If~$f$ is incident to a
  single connected component of~$G$, and the bridge forest of this
  component is a single node, a realization of~$A$ for~$f$, which
  exists by Theorem~\ref{thm:biconn-aug}, forms the desired matching.
  Otherwise, the conditions satisfied by~$A$ imply that we can find
  enough vertices of~$\Vbound$ assigned to~$f$ and match them to the
  dummy vertices associated with~$f$.  The choice of their demands and
  the parity condition imply that the number of unmatched vertices
  in~$\Vbound$ assigned to~$f$ is even, and they can be paired
  arbitrarily in~$G_A''$ to form~$M_f$.
\end{proof}

To decide biconnected \faug for a given maxdeg-3 graph on~$n$ vertices with a
fixed embedding, we thus first construct in~$O(n^2)$ time the
biconnectivity assignment graph~$G_A''$ and compute in~$O(n^{2.5})$
time a perfect (generalized) matching in it. If such a matching does not
exist, then a biconnected augmentation does not exist by
Theorem~\ref{thm:biconn-aug} (i).  Otherwise, such a matching induces a
node assignment satisfying the parity condition, the matching
condition, and the biconnectivity condition by
Lemma~\ref{lem:biconnectity-assignment-graph}.  Using
Corollary~\ref{cor:biconnected-planar-realizable}, we modify it
in~$O(n)$ time to a node assignment additionally satisfying the
planarity condition.  A corresponding augmentation can then be found
in~$O(n^2)$ time by Theorem~\ref{thm:biconn-aug}(ii).
  
\begin{theorem}
\label{thm:biconnectedFERA}
Biconnected \faug can be solved in~$O(n^{2.5})$ time.
\end{theorem}

% ------ Proof of Completeness of triconnected FERA ---------
\section{Proof of Completeness of Triconnected \faug}
\label{sec:tricon}
\newcommand{\theoremTc}{Triconnected \faug is NP-complete,
even if the input graph is already biconnected.
%\comment{NP-Notation in mathcal?}
}
\begin{theorem}
  \label{thm:tc}
  \theoremTc
\end{theorem}
\begin{proof}
  Triconnected \faug is in~$NP$ since, given a planar graph~$G$ with a
  fixed embedding, we can guess a set~$W \subseteq \binom{V}{2}$ of
  non-edges of~$G$ and then test efficiently whether the graph~$G+W$
  is 3-regular, planar, and triconnected, and that $W$ respects the
  given embedding of~$G$ (the latter can be checked using an algorithm
  due to Angelini et al.~\cite{adfjk-tppeg-10}).  We prove NP-hardness
  by reducing from the problem \mpsat, which is known to be
  NP-hard~\cite{bk-obsp-10}. It is a special variant of \psat, which
  we use in the next section for the hardness proof of \planaug.
    \begin{figure}[t]
    \centering
    \includegraphics{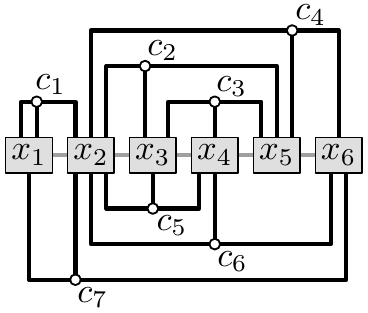}
    \caption{Layout of a (monotone) planar \sat formula.}
    \label{fig:sat-layout}
  \end{figure}
  A monotone planar \sat formula
  %An instance of \mpsat\comment{eventuell Beschreibung von psat darauf beziehen?} 
  is a \sat formula whose clauses either contain only positive or
  negative literals and whose variable--clause graph is planar.  A
  \emph{monotone rectilinear representation} of a monotone planar \sat
  formula is a drawing of the variable--clause graph
  %can be laid out 
  %(in polynomial time) 
  such that the variables correspond to axis-aligned rectangles 
  on the x-axis and clauses correspond to
  non-crossing three-legged ``combs'' above the x-axis if they contain positive variables and below the
  x-axis otherwise; see Fig.~\ref{fig:sat-layout}.  
  An instance of \mpsat is a monotone rectilinear representation of a monotone planar \sat formula~$\varphi$.
 % \comment{Wie  wird Darstellbarkeit belegt?}
  We now construct a 
  biconnected 
  graph~$G_\varphi$ with a fixed planar embedding that admits a
  planar 3-regular triconnected augmentation if and only if~$\varphi$ 
  is satisfiable.

  The graph~$G_\varphi$ consists of so-called \emph{gadgets}, that is
  subgaphs that represent the variables, literals, and clauses
  of~$\varphi$; see Fig.~\ref{fig:tricongadget}.  
  For each gadget, we will
  argue that there are only a few ways to augment it to be
  3-regular, triconnected and planar.  
  Note that our construction connects variable
  gadgets corresponding to neighboring variables in the layout of the
  variable--clause graph of~$\varphi$.  Hence~$G_\varphi$ is always
  connected.  Additionally, we identify the left boundary of the
  leftmost variable gadget with the right boundary of the rightmost
  variable gadget.  In the figure vertices with degree less than~3 are
  highlighted by white disks.  All bends and junctions of line
  segments represent vertices of degree at least~3.  Vertices of degree
  greater than~3 are actually modeled by small cycles of vertices of
  degree~3, as indicated in the left of Fig.~\ref{fig:tricongadget}.  The
  (black thick and thin) solid line segments between adjacent vertices
  represent the edges of~$G_\varphi$; the dotted line segments
  represent non-edges of~$G_\varphi$ that are candidates for an
  augmentation of~$G_\varphi$. 
  Gaps in the thick black line segments of the literal gadgets indicate positions where further 
  subgraphs can be plugged in depending on the number of clauses 
  containing the literal.
   
  Each variable gadget consist of two symmetric parts, which
  correspond to the two literals. These literal (sub)gadgets are
  separated by (thick) horizontal edges. The degree-2 vertex $u$ is
  incident to both literal gadgets. The thin triangle at the right
  side is called the \emph{parity triangle} (see
  Fig.~\ref{fig:tricongadget}).
  Each literal gadget 
  contains a subgraph that is attached to the horizontal edges separating the 
  literals in only two vertices, which thus form a separator of size~2.
  We call this subgraph the \emph{literal body}. 
  The literal body can be considered as a path of 
  smaller (thin) subgraphs connected by thick black edges. 
  The thin subgraphs can be characterized as a triangle at the front side (\emph{front triangle})
  that is based on another triangular shaped subgraph (\emph{triangle basement}) and further
  oppositely placed \emph{pairs of triangles}.
  In Fig.~\ref{fig:tricongadget} we exemplarily marked a front triangle with its 
  triangle basement and a pair of triangles.
  In the construction the number of 
  pairs of triangles in the literal body corresponds to the 
  number of clauses containing the literal. Note that w.l.o.g.\
  we may assume that each literal appears in at least one clause.
  The necessary number of pairs of triangles can be plugged in at the gap. 
  The corresponding clauses are attached to the outer boundary of the literal gadget, 
  as exemplarily shown in Fig.~\ref{fig:tricongadget}.
  Each attached clause thereby induces a pair of adjacent degree-2 vertices at 
  the boundary that are incident to the literal gadget and to the clause gadget.
  We call the corresponding valencies the \emph{boundary valencies} of the literal gadget.
  Thus, each literal gadget has twice as many boundary valencies
  as clauses contain the literal.
  
  Consider the graph $G_\varphi'$ that we obtain by deleting the literal bodies, contracting 
  the parity triangles and ignoring degree-2 vertices. 
  We claim that $G_\varphi'$ is 3-vertex connected. 
  This is true since (a) the subgraph of $G_\varphi'$ induced by the 
  variable gadgets is 3-connected and (b) each subgraph induced by a clause gadget is 
  also 3-connected and is attached to the former (variable gadget) subgraph in twelve vertices.
  Hence, a 3-regular triconnected augmentation of $G_\varphi$ only needs to care for the connectivity at the literal bodies 
  and the parity triangles. 
  Note that $G_\varphi$ is already biconnected since it is obtained from a 3-connected 
  graph by subdividing edges, replacing degree-2 vertices by (parity) triangles and 
  adding paths of biconnected subgraphs (literal bodies) between existing endpoints.
   In the following we call a 3-regular, triconnected, planar augmentation a \emph{valid} augmentation.
  We show two properties of $G_\varphi$:
    \begin{enumerate}[(P1)]
  \itemsep=0ex
  \item Let $W$ denote a valid augmentation and $x$ a variable gadget. Then for 
  at least one literal gadget in $x$ the augmentation $W$ assigns all boundary valencies 
  to the incident literal face. \label{prop:allVal}
  \item Given a literal $L$ and a (sub)set of clauses containing $L$, there exists a 
  valid augmentation of the corresponding variable gadget that uses all 
  boundary valencies of $L$ 
  apart from those that are incident to the given clauses. \label{prop:evenVal}
   \end{enumerate}
    \begin{figure}[t]
    \centering
    \includegraphics{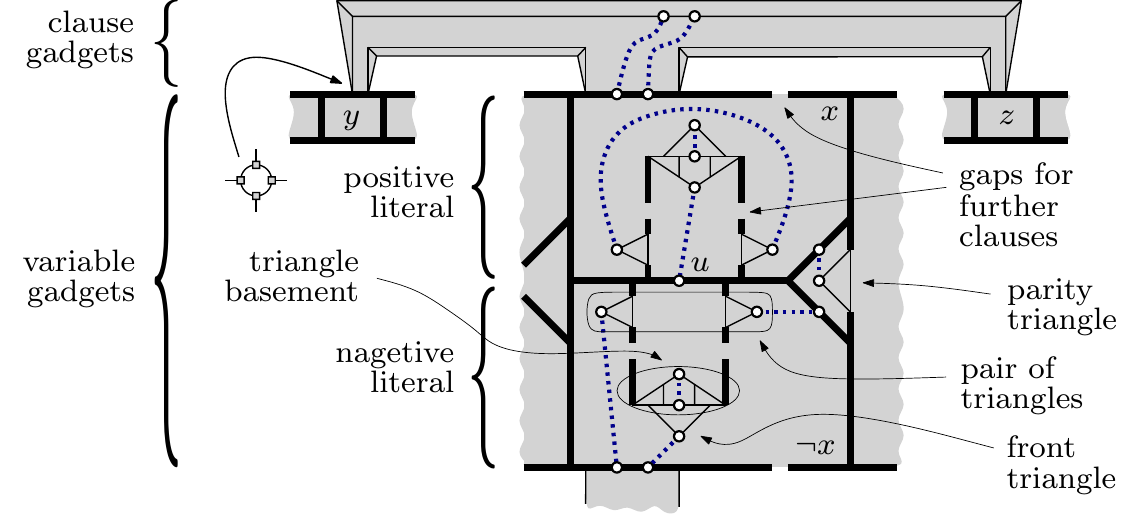}
    \caption{Variable gadget for variable $x$ and clause gadget for 
       clause $(y \vee x \vee z)$ in graph~$G_\varphi$.  The
      augmentation (dotted edges) corresponds to the assignment
      $x=$\emph{ture}, $\neg x=$\emph{false}.}
    \label{fig:tricongadget}
  \end{figure}

   We start with (P\ref{prop:allVal}). Consider the exemplary variable
   gadget in Fig.~\ref{fig:tricongadget}.  The valency of $u$ is
   incident to both literal (sub)gadgets, and hence, is either
   assigned to $x$ or $\neg x$ by a valid augmentation. Without loss
   of generality, assume that $u$ is assigned to $x$.  The opposite
   case is symmetric.  The two degree-2 vertices in the triangle
   basement in $\neg x$ are thus connected since the inner face of
   the literal body provides no further valencies.
   Let $\ell$ denote the number of clauses containing $\neg x$. 
   The outer face of the literal body of $\neg x$ is incident to $2(\ell +1)$ valencies; 
   $\ell +1$ stem from the triangles at the literal body, $\ell$ are boundary 
   valencies and one additional valency is placed at the triangle to the right.
   We argue that the valencies at the triangles of the literal body are not connected to each other 
   by a valid augmentation.
   This is true since such an edge would immediately induce a subgraph 
   that is separated from the rest by only two 
   vertices; namely the vertices where the connected triangles are attached to the literal body.
   Consequently, a valid augmentation must assign all $\ell$ boundary valencies of $\neg x$ to the literal face.
   The last valency, which is necessary due to the parity condition, is provided by the vertex at the thick triangle to the right.
   
   For the proof of (P\ref{prop:evenVal}) consider again
   Fig.~\ref{fig:tricongadget} and let (without loss of generality)
   $x$ denote the given literal.  The number of clauses containing $x$
   is $\ell$, $0 \leq s \leq \ell$ denotes the cardinality of the
   given subset of the clauses containing $x$.  In order to construct
   a valid augmentation $W$ of the variable gadget such that $W$ uses
   exactly $2(\ell-s)$ boundary valencies of $x$, we connect $u$ to a
   valency in $x$.  This induces an augmentation of $\neg x$ as
   described in the proof of (P\ref{prop:allVal}).  Note that this
   augmentation makes the triangle basement in $\neg x$ triconnected
   and all the triangles of the literal body are connected to vertices
   outside the literal body, which also makes the literal body
   triconnected.  In the literal gadget for $x$ the only vertex that
   can be connected to $u$ belongs to the triangle basement.  Hence,
   the two remaining degree-2 vertices at the front side of the
   literal body are also connected.
   Furthermore, we connect the valency at the parity triangle to the only possible vertex at the opposite thick edge, 
   which makes the parity triangle triconnected.
   Finally, we choose the $s$ upper pairs of triangles at the literal body and connect each by an edge.
   In contrast to the proof of (P\ref{prop:allVal}) connecting opposite triangles at the literal body is feasible,
   since the new edge incident to $u$ ensures triconnectivity.
   The remaining $2(\ell - s)$ valencies at the literal body 
   can be obviously connected in a planar way to the $2(\ell - s)$ boundary 
   valencies that are not incident to the given clauses, which finally 
   ensures the triconnectivity of the whole augmented variable gadget.
   
   With the help of (P\ref{prop:allVal}) and (P\ref{prop:evenVal}) it
   is now easy to show that if $G_\varphi$ admits a valid augmentation
   then $\varphi$ is satisfiable.  Assume that~$W$ is a valid
   augmentation. Then $W$ connects the two degree-2 vertices of each
   clause to two boundary valencies of literal gadgets since
   connecting those degree~2 to each other would yield a parallel
   edge.
   This \emph{selects} a set of literal gadgets in that sense that a
   gadget is selected if at least one of its boundary valencies is
   assigned to a clause face.  According to (P\ref{prop:allVal}) the
   boundary valencies of the negated literal gadget of a selected
   gadget are all assigned to the literal face, and hence, a literal
   and its negation are never selected at the same time. Thus, the
   literal selection induces a truth assignment of the variables,
   which satisfies $\varphi$ since each clause selects at least one
   (true) literal.
   
   Conversely, we need to show that if $\varphi$ is satisfiable then
   $G_\varphi$ admits a valid augmentation.  Assume we have a
   satisfying truth assignment for $\varphi$. For each clause, we
   choose exactly one true literal~$L$ and connect the two degree-2
   vertices of the clause to the two boundary valencies of $L$ that
   are incident to the clause gadget.  
   This ensures triconnectivity at the former degree-2 vertices of the clause 
   and the former degree-2 vertices providing the boundary valencies, and thus,
   yields a valid augmentation of the clause gadgets. Recall that $G'_\varphi$ is already triconnected. 
  With the help of (P\ref{prop:evenVal}) this can be finally extended to a valid augmentation of $G_\varphi$. 
\end{proof}
% --------------------------------------------------------------------------

%\vspace{-1ex}
\section{Proof of Completeness of \planaug}
\label{sec:hardness}
\rephrase{Theorem}{\ref{thm:npc}}{\theoremNpc}

\begin{proof}
  \planaug is in~$NP$ since given a planar graph~$G$ we can guess a
  set~$W \subseteq \binom{V}{2}$ of non-edges of~$G$ and then test
  efficiently whether~$G+W$ is 3-regular and planar.  We prove
  NP-hardness by reducing from the problem \psat, which is known to be
  NP-hard~\cite{l-pfu-82}.  The reduction is inspired by and indeed
  very similar to a reduction of Rutter and Wolff~\cite{rw-acpgg-12},
  showing that it is NP-hard to find a smallest edge set that augments
  a given graph to be 2-edge connected and planar.

%  \begin{figure}[tb]
%    \centering
%    \includegraphics{planar_formula}
%    \caption{Layout of a planar \sat formula.}
%    \label{fig:sat-layout}
%  \end{figure}

  An instance of \psat is a \sat formula~$\varphi$ whose
  variable--clause graph is planar.  Such a graph can be laid out (in
  polynomial time) such that the variables correspond to pairwise
  %\comment{weshalb paarweise?}
  axis-aligned rectangles on the x-axis and clauses correspond to
  non-crossing three-legged ``combs'' above or below the
  x-axis~\cite{kr-pcr-92}; see Fig.~\ref{fig:sat-layout}.  We now
  construct a biconnected planar graph~$G_\varphi$ that admits a
  planar 3-regular augmentation if and only if~$\varphi$ has a
  satisfying truth assignment.

  \begin{figure}[tb]
    \centering
    \includegraphics{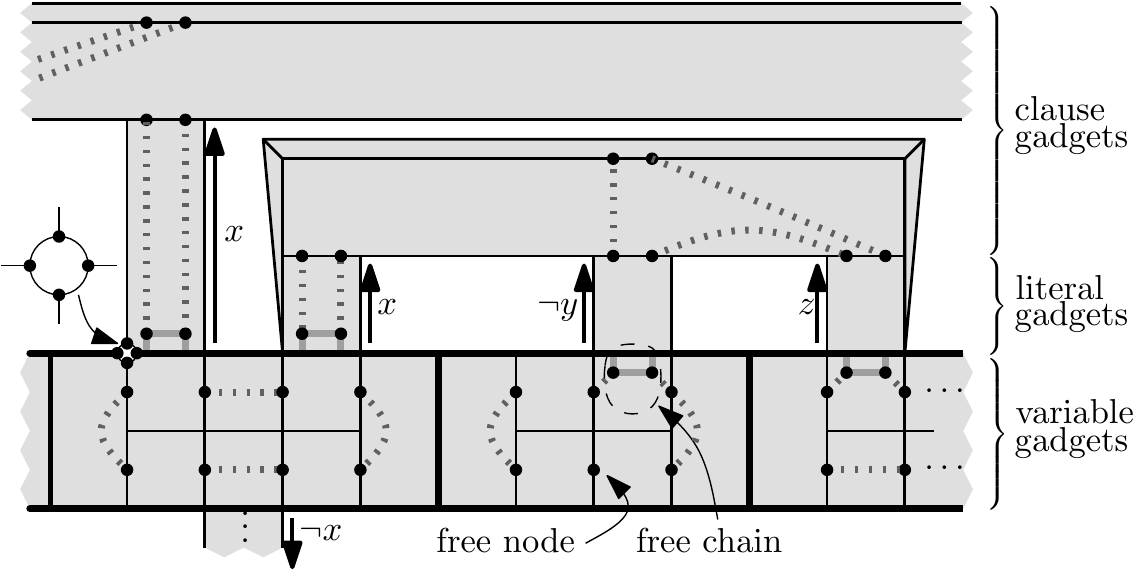}
    \caption{Part of the graph~$G_\varphi$ for a SAT formula~$\varphi$
      that contains the clause $(x \vee \neg y \vee z)$.  The
      augmentation (dotted edges) corresponds to the assignment
      $x=y=$\emph{false} and~$z=$\emph{true}.}
    \label{fig:gadget}
  \end{figure}

  The graph~$G_\varphi$ again consists of \emph{gadgets}, which are subgaphs that represent the variables, literals, and clauses
  of~$\varphi$; see Fig.~\ref{fig:gadget}.  For each gadget, we will
  argue that there are only a few ways to embed and augment it to be
  3-regular and planar.  Note that our construction connects variable
  gadgets corresponding to neighboring variables in the layout of the
  variable--clause graph of~$\varphi$.  Hence~$G_\varphi$ is always
  connected.  Additionally, we identify the left boundary of the
  leftmost variable gadget with the right boundary of the rightmost
  variable gadget.  In the figure vertices with degree less than~3 are
  highlighted by small black disks.  All bends and junctions of line
  segments represent vertices of degree at least~3.  Vertices of degree
  greater than~3 are actually modeled by small cycles of vertices of
  degree~3, as indicated in the left of Fig.~\ref{fig:gadget}.  The
  (black and dark gray) solid line segments between adjacent vertices
  represent the edges of~$G_\varphi$; the thick dotted line segments
  represent non-edges of~$G_\varphi$ that are candidates for an
  augmentation of~$G_\varphi$.  The set of solid black edges forms a
  subgraph of~$G_\varphi$ that we call the \emph{frame}.  The dark
  gray solid edges form \emph{free chains}, which connect two degree-2
  vertices to the frame.  Consider the graph~$G_\varphi'$ that is
  obtained from the frame by contracting all vertices of degree~2
  and all cycles that are used to model vertices of degree greater
  than~3.  The graph~$G_\varphi'$ coincides with the one used by
  Rutter and Wolff in their reduction~\cite{rw-acpgg-12}, and they show
  that it is 3-connected, and thus has a unique planar
  embedding~\cite{w-cgcg-32}.

  Since subdividing edges and replacing a vertex of degree at least~4 by
  a cycle preserves 3-connectedness, the frame has a unique embedding
  as well.  In other words, the embedding of~$G_\varphi$ is fixed up
  to embedding the free chains, which may be embedded in two distinct
  faces, each.

  A 3-regular planar augmentation of~$G_\varphi$ yields an embedding
  of~$G_\varphi$ and an assignment of the degree-2 vertices
  of~$G_\varphi$ to incident faces (a vertex~$v$ is assigned to the
  face~$f$ if in the planar embedding of~$G+W$ the edge of~$W$
  incident to~$v$ is embedded in the (former) face $f$) such that
  \begin{enumerate}[(P1)]
  \itemsep=0ex
  \item each face is assigned an even number of vertices,\label{prop:even}
  \item each face that is assigned two adjacent vertices is assigned
    at least four vertices. \label{prop:chain}
  \end{enumerate}
  We call such an assignment of degree-2 vertices to faces
  \emph{valid}.  Conversely, it is readily seen that given a valid
  assignment, a planar 3-regular augmentation can always be
  constructed.  We thus need to show that if~$G_\varphi$ admits an
  embedding with a valid assignment, then~$\varphi$ is satisfiable.

  Our variable gadget consists of two rows of square faces where the
  horizontal edge between the two leftmost faces and the horizontal edge
  between the two rightmost faces is missing.  Effectively, the inner
  faces of a variable box form a cycle.  Starting from the leftmost
  (rectangular) face, we call the faces \emph{odd} and \emph{even}.
  Each interior vertical edge is subdivided by a degree-2 vertex.  Due
  to property~(P\ref{prop:even}), these subdivision vertices must either
  all be assigned to the odd faces or all to the even faces of the
  variable.  If the vertices are assigned to the even faces, then the
  corresponding variable is true, and vice versa.

  A literal gadget consists of a square face that lies immediately
  above or below the variable gadget.  A positive literal (such as the
  one labeled with~$x$ in Fig.~\ref{fig:gadget}) is attached to an
  even face, a negated literal (such as the one labeled with~$\neg y$
  in Fig.~\ref{fig:gadget}) is attached to an odd face.  A literal
  gadget contains two adjacent subdivision vertices at the
  edge it shares with the clause gadget, and
  a free chain containing two adjacent vertices of
  degree~2.  The latter is attached to the boundary shared by the literal gadget with
  the variable gadget.  Due to
  property~(P\ref{prop:chain}) the free chain must either be embedded
  inside the literal gadget and all incident degree-2 vertices are
  assigned to the face of the literal gadget, or the chain is embedded inside
  the attached variable gadget and the two subdivision vertices are assigned
  to the adjacent clause gadget.  Again due to
  property~(P\ref{prop:chain}) the free chain must be embedded inside
  the literal gadget if no vertices are assigned to the adjacent face
  of the variable gadget.  In this case the literal has the value
  \emph{false}.  If two vertices are assigned to the adjacent face of
  the variable gadget, the free chain can (but does not have to) be
  embedded inside the variable and the two subdivision vertices are
  assigned to the clause.

  Finally, each clause gadget consists of a single rectangular face
  that contains two adjacent subdivision vertices.  If~$G_\varphi$
  admits an embedding with a valid assignment, then, due to
  property~(P\ref{prop:chain}), at least two other degree-2 vertices are
  assigned to the clause gadget face.  This means that for each clause
  gadget, the two subdivision vertices of at least one literal are
  assigned to the clause gadget.  In other words, at least one of the
  literals that make up the clause is \emph{true}.  Hence,~$\varphi$
  has a satisfying truth assignment.

  Conversely, it is easy to see that if~$\varphi$ has a satisfying
  truth assignment, then an embedding with a corresponding assignment
  can be found.  We use a constant number of vertices and edges for
  each literal and clause gadget, thus our reduction---including the
  computation of the embedding of the variable--clause graphs---is
  polynomial.  Moreover, since the graph~$G_\varphi$ is obtained from
  a 3-connected graph by subdividing edges and adding some paths
  between existing endpoints, the graph~$G_\varphi$ is biconnected.
\end{proof}

\section{Conclusion}
\label{sec:conclusion}
%\vspace{-1ex}
In this paper we have given efficient algorithms for deciding whether a given planar
graph with a fixed embedding admits a 3-regular planar augmentation.
We note that the running time of $O(n^{2.5})$ is due to the potentially 
quadratic size of our assignment graphs. Recently, we succeeded in 
constructing equivalent assignment graphs with only $O(n)$ edges. 
This immediately improves the running time of all our algorithms to $O(n^{1.5})$.

 \bibliographystyle{plain}
 \bibliography{cubicPlanarAug_arXiv}

\end{document}